\PassOptionsToPackage{sort&compress}{natbib}
\documentclass[final,3p,11pt,nopreprintline]{elsarticle}
\usepackage[T1]{fontenc}
\usepackage[utf8]{inputenc}
\usepackage{amsmath,amsthm}
\usepackage{newtxtext,newtxmath}
\usepackage{microtype}
\usepackage{graphicx,subcaption,booktabs,array,enumitem,xcolor}
\graphicspath{{./}}
\newcolumntype{L}[1]{>{\raggedright\arraybackslash}p{#1}}
\definecolor{linkblue}{RGB}{0,115,152}
\usepackage[colorlinks=true,linkcolor=linkblue,citecolor=linkblue,urlcolor=linkblue]{hyperref}

\usepackage{fancyhdr}
\makeatletter
\renewenvironment{proof}[1][\proofname]{\par
  \pushQED{\qed}%
  \normalfont \topsep6\p@\@plus6\p@\relax
  \trivlist
  \item[\hskip\labelsep\bfseries #1\@addpunct{.}]\ignorespaces
}{%
  \popQED\endtrivlist\@endpefalse
}
\makeatother

\newcommand{\eps}{\varepsilon}
\newcommand{\sgn}{\operatorname{sgn}}
\newcommand{\Ai}{\operatorname{Ai}}
\newcommand{\Int}{\operatorname{Int}}
\newcommand{\cH}{\mathcal{H}}
\newcommand{\cP}{\mathcal{P}}
\newcommand{\cE}{\mathcal{E}}

\newcommand{\Fe}{F_{\mathrm{e}}}
\newcommand{\Fo}{F_{\mathrm{o}}}

\newtheorem{theorem}{Theorem}[section]
\newtheorem{lemma}[theorem]{Lemma}
\newtheorem{proposition}[theorem]{Proposition}
\newtheorem{corollary}[theorem]{Corollary}
\newtheorem{conjecture}[theorem]{Conjecture}
\newtheorem*{openproblem}{Open Problem}
\theoremstyle{definition}
\newtheorem{definition}[theorem]{Definition}
\theoremstyle{remark}
\newtheorem{remark}[theorem]{Remark}
\numberwithin{equation}{section}
\setlist{leftmargin=2em}
\allowdisplaybreaks[1]

\newcommand{\Panel}[5][fig]{%
  \begin{subfigure}[t]{#2\textwidth}
    \centering
    \includegraphics[width=\linewidth]{#1/#3.pdf}
    \caption{#4}\label{#5}
  \end{subfigure}}

\newcommand{\PlotPanel}[3]{%
  \begin{subfigure}[t]{.31\textwidth}
    \centering
    \includegraphics[scale=.93]{#1.pdf}
    \caption{#2}\label{#3}
  \end{subfigure}}
\begin{document}
\begin{frontmatter}

\title{Slow--fast quintic Li\'enard systems admit at most 2 limit cycles}

\author[a]{Hebai Chen}
\ead{chen_hebai@csu.edu.cn}
\author[a]{Leshan Chen}
\ead{chenleshan@csu.edu.cn}
\author[a]{Yu Xiao}
\ead{xy.math@csu.edu.cn}
\affiliation[a]{organization={School of Mathematics and Statistics, Central South University},
            city={Changsha},
            postcode={410083},
            state={Hunan},
            country={China}}

\begin{abstract}
For each fixed real quintic polynomial $F$, we prove that the slow--fast Li\'enard system $\dot x=y-F(x)$, $\dot y=-\eps(x-\lambda)$ has at most two limit cycles, counted with multiplicity, for all sufficiently small $\eps>0$ and $|\lambda|$. Any nonhyperbolic periodic orbit is unique and has multiplicity exactly two. Whereas Li and Lu (2014) established a local cyclicity bound of two for each
nondegenerate slow--fast cycle, our bound is global and simultaneously accounts
for limit cycles approaching distinct slow--fast cycles, shrinking to the equilibrium, or
accumulating on degenerate slow--fast cycles not covered by the local theory. The proof combines a classification of the limit periodic sets, the slow
divergence integral, an analysis of the transition maps near slow--fast cycles with vanishing slow
divergence integral, and, in the degenerate cases, a three-scale analysis of the cycles that shrink
to the equilibrium.
\end{abstract}


\end{frontmatter}

\section{Introduction and main results}\label{sec:introduction}

The problem of determining the number of limit cycles of polynomial Li\'enard systems
\begin{equation}\label{eq:lienard}
\dot x=y-F(x),\qquad \dot y=-x,
\end{equation}
is an important special case of Hilbert's sixteenth problem. Let $L(n)$ denote the supremum of the numbers
of limit cycles among all systems of the form~\eqref{eq:lienard} with $\deg F=n$.
In 1977, Lins Neto, de Melo, and Pugh established the lower bound
$L(n)\ge\lfloor(n-1)/2\rfloor$ and conjectured that equality holds~\cite{LinsMeloPugh1977}.
The broader challenge of obtaining quantitative upper bounds for the number of limit cycles
of planar polynomial vector fields, including those of Li\'enard type, was subsequently
emphasized by Smale~\cite{Smale1991,Smale1998}. For degrees up to four, the sharp bounds are known:
$L(n)=0$ for $n\le2$ and $L(3)=L(4)=1$, with the quartic case resolved by Li and
Llibre~\cite{LiLlibre2012}. The quintic case, however, has been substantially more
challenging. Rychkov established an upper bound of two limit cycles for odd quintic
Li\'enard systems~\cite{Rychkov1975}, while the classical lower bound $L(5)\ge2$ is
attained. For an account of the historical development and the general quintic upper-bound
problem, we refer to~\cite{LlibreZhang2017}.

Dumortier, Panazzolo and Roussarie first disproved the Lins--de Melo--Pugh conjecture
in higher degrees, constructing four limit cycles for a degree-seven Li\'enard
system~\cite{Dumortier2007}. De Maesschalck and Dumortier subsequently extended the
counterexamples to degree six~\cite{DeMaesschalckDumortier2011}.
The quintic case therefore lies between the sharp low-degree results and these
higher-degree counterexamples. These constructions illustrate the role of slow--fast
dynamics in limit-cycle bifurcations. Fenichel theory describes normally hyperbolic slow
manifolds~\cite{Fenichel1979}, while blow-up and slow-divergence methods analyze folds
and canard cycles~\cite{DumortierRoussarie1996,KS,KrupaSzmolyan2001Relaxation}.
Further work addresses entry--exit relations and canard cycles~\cite{DeMaesschalckDumortier2005,DeMaesschalckDumortier2008,DumortierRoussarie2001}, while slow divergence integrals have been studied for balanced canards and near centers~\cite{Dumortier2011,DeMaesschalckHuzak2015}.
Related studies of canard phenomena and the cyclicity of slow--fast cycles in predator--prey
systems include~\cite{LiZhu2013,WangZhang2019,YaoEtAl2024}. For stability-loss delay and
the smoothness of return maps, see~\cite{WangZhang2018}.
Other contributions concern slow--fast saddle-node bifurcations~\cite{Huzak2017} and quartic Li\'enard equations with linear damping~\cite{Huzak2019}.

Let $F$ be a real polynomial of degree five and consider the Li\'enard system with large damping
\begin{equation}\label{eq:large}
\dot x=y-\mu F(x),\qquad \dot y=-(x-\lambda),\qquad \mu\gg1 .
\end{equation}
With $Y=y/\mu$, the new time $s=\mu t$ and $\eps=\mu^{-2}$, system~\eqref{eq:large} becomes, after
renaming $Y$ and $s$ as $y$ and $t$,
\begin{equation}\label{eq:system}
\dot x=y-F(x),\qquad \dot y=-\eps(x-\lambda).
\end{equation}
Here $\eps>0$ is small and $\lambda$ is close to $0$. The translation $y\mapsto y-F(0)$ and, if
necessary, the change $(t,y,F)\mapsto(-t,-y,-F)$ do not change the number of periodic orbits (the
second one reverses their stability). We may therefore assume that
\begin{equation}\label{eq:F}
F(x)=a_1x+a_2x^2+a_3x^3+a_4x^4+a_5x^5,\qquad a_5>0 .
\end{equation}
Throughout, $F$ is fixed, and $(\eps,\lambda)\to(0^+,0)$ without any restriction on the relative
rates. Let $N_F(\eps,\lambda)$ denote the number of periodic orbits of system~\eqref{eq:system}. By
Lemma~\ref{lem:isolated} below every periodic orbit of system~\eqref{eq:system} is isolated, so
$N_F(\eps,\lambda)$ is the number of limit cycles.

Li and Lu~\cite[Theorems~1.1 and~3.1]{LiLu} bounded the cyclicity of each nondegenerate
slow--fast cycle by two, and by one when its slow divergence integral is nonzero. Their
nondegeneracy condition requires quadratic folds and excludes transitory canards.
For transitory canards they used the jump-point estimate of
\cite[Theorem~1]{DeMaesschalckDumortierRoussarie2014}.
For shape~E in their Figure~12 and shapes~(d2), (e2) and~(g2) in their Figure~14, they obtained
the bound three and asked whether it could be reduced to two~\cite[Remark~3.2]{LiLu}.
Theorem~\ref{thm:main} counts all periodic orbits for a fixed quintic, including sequences
approaching different slow--fast cycles or shrinking to the equilibrium.

\begin{theorem}\label{thm:main}
For every real polynomial $F$ of degree five, there exists
$\eps_0=\eps_0(F)>0$ such that, whenever $0<\eps<\eps_0$
and $|\lambda|<\eps_0$, system~\eqref{eq:system} has at most
two limit cycles, counted with multiplicity. Here the multiplicity of a limit cycle is the order of the
corresponding zero of the local Poincar\'e return displacement.
In particular,
\[
N_F(\eps,\lambda)\le2.
\]
If a nonhyperbolic periodic orbit exists, it is the unique
periodic orbit and has multiplicity exactly two.
\end{theorem}

\begin{remark}\label{rem:sharp}
The bound in Theorem~\ref{thm:main} is sharp. For $F_*(x)=3x^5-50x^3+135x$, we have
$F_*'(x)=15(x^2-1)(x^2-9)$ with extrema $F_*(\pm1)=\pm88$ and $F_*(\pm3)=\mp216$. At $\lambda=0$, the
critical curve $y=F_*(x)$ carries two strongly common slow--fast cycles surrounding the origin: an outer
attracting one and an inner repelling one (Figure~\ref{fig:intro}). Applying
\cite[Theorem~2]{DeMaesschalckDumortierRoussarie2011} in forward time to the outer cycle and in
backward time to the inner cycle yields, for all sufficiently small $\eps>0$ and $|\lambda|$, two
hyperbolic limit cycles (outer attracting, inner repelling), so $N_{F_*}(\eps,\lambda)=2$ by
Theorem~\ref{thm:main}.
\end{remark}

\par\smallskip\noindent
\begin{minipage}{\linewidth}
\captionsetup{type=figure,skip=4pt}
\captionsetup[subfigure]{skip=2pt}
\centering
\Panel[fig]{.39}{intro-a}{The two common slow--fast cycles}{fig:intro-a}\hspace{.04\linewidth}%
\Panel{.39}{intro-b}{The two limit cycles for $\eps=0.1$}{fig:intro-b}
\caption{Sharpness example of Remark~\ref{rem:sharp} at $\lambda=0$. In (a), blue and red mark attracting and repelling slow arcs and black marks fast jumps; in (b), the inner and outer limit cycles cross the positive $y$-axis at $y\approx88.09876$ and $y\approx216.14789$.}
\label{fig:intro}
\end{minipage}\par\smallskip

The bound in Theorem~\ref{thm:main} can be sharpened for several configurations of the critical
curve in Figure~\ref{fig:classification}: for all sufficiently small $\eps>0$ and $|\lambda|$,
there are no periodic orbits when $F$ has no local extrema (panel~(k)), and at most one when the
limiting equilibrium is a quartic extremum (panel~(h)). The example in Figure~\ref{fig:intro}
belongs to case~(j), where the bound of two is attained. In panels~(d)--(g), $A$, $B$, and $C$
denote the heights of the three extrema other than the minimum at the origin, ordered by
increasing $x$; the corresponding bounds for all configurations are summarized in
Table~\ref{tab:cases} (Section~\ref{sec:min-setup}).

\begin{figure}[!htb]
\centering
\Panel{.235}{classification-a}{Outer minimum:\\four folds}{fig:classification-a}\hspace{.013\textwidth}%
\Panel{.235}{classification-b}{Outer minimum:\\two extrema}{fig:classification-b}\hspace{.013\textwidth}%
\Panel{.235}{classification-c}{Remote quartic\\extremum}{fig:classification-c}
\par\vspace{10pt}
\Panel{.235}{classification-d}{Inner minimum:\\$A<B$}{fig:classification-d}\hspace{.013\textwidth}%
\Panel{.235}{classification-e}{Inner minimum:\\$A=B$}{fig:classification-e}\hspace{.013\textwidth}%
\Panel{.235}{classification-f}{Inner minimum:\\$A>B$, $C>0$}{fig:classification-f}\hspace{.013\textwidth}%
\Panel{.235}{classification-g}{Inner minimum:\\$A>B$, $C\le0$}{fig:classification-g}
\par\vspace{10pt}
\Panel{.235}{classification-h}{Quartic minimum\\at $x=0$}{fig:classification-h}\hspace{.013\textwidth}%
\Panel{.235}{classification-i}{Locally monotone:\\two extrema}{fig:classification-i}\hspace{.013\textwidth}%
\Panel{.235}{classification-j}{Locally monotone:\\four extrema}{fig:classification-j}\hspace{.013\textwidth}%
\Panel{.235}{classification-k}{Monotone\\quintic}{fig:classification-k}
\caption[Representative quintic critical curves.]{Representative quintic critical curves.
Solid branches attract ($F'>0$), dashed branches repel ($F'<0$), and dots mark the limiting equilibrium.
}
\label{fig:classification}
\end{figure}

\begin{remark}\label{rem:lilu}
Time reversal identifies system~(3.2) of~\cite{LiLu} with~\eqref{eq:system} for fixed $F$
and $\lambda=\lambda(\eps)$. A translation moves any limiting $\lambda_0$ to zero.
Theorem~\ref{thm:main} answers the question in
\cite[Remark~3.2]{LiLu} for this fixed-polynomial family, with multiplicities as in their
Definition~1.1. Shape~E is the equal-maximum configuration in Section~\ref{sec:equal-max}.
The three shapes in their Figure~14 belong to the configuration in Figure~\ref{fig:classification}(j), where a mixed
limit requires fold resonance and Proposition~\ref{prop:res} handles the vanishing integral.
The configuration in Figure~\ref{fig:classification}(h) covers the quartic turning point at the
equilibrium, which was excluded in the proof of~\cite[Lemma~3.2]{LiLu}, and yields the upper
bound one. We do not consider perturbations of the coefficients of $F$ here.
\end{remark}

\begin{corollary}\label{cor:uniform-lambda}
For every real polynomial $F$ of degree five, there exists $\eps_1=\eps_1(F)>0$ such that, for every
$0<\eps<\eps_1$ and every $\lambda\in\mathbb R$, system~\eqref{eq:system} has at most two limit
cycles, counted with multiplicity. Moreover, any nonhyperbolic periodic orbit is unique and has
multiplicity exactly two.
\end{corollary}

\begin{proof}
Normalize $F$ as in~\eqref{eq:F} with $a_5>0$. Since $F'(\lambda)=5a_5\lambda^4+O(|\lambda|^3)$ and
$F'''(\lambda)=60a_5\lambda^2+O(|\lambda|)$, there exists $R=R(F)>0$ such that $F'(\lambda)>0$ and
$F'''(\lambda)>0$ whenever $|\lambda|\ge R$. For such $\lambda$, the odd part~\eqref{eq:odd-part}
satisfies $\Fo(W)/z=F'(\lambda)+\frac16F'''(\lambda)z^2+a_5z^4>0$ for all $z=\sqrt{2W}>0$, so
Lemma~\ref{lem:odd} below excludes periodic orbits of any amplitude for every $\eps>0$.

For each $q\in[-R,R]$, shifting coordinates to $F_q(u)=F(q+u)-F(q)$ and applying Theorem~\ref{thm:main}
yields a threshold $\eps_0(F_q)>0$ such that the stated multiplicity bound holds whenever
$0<\eps<\eps_0(F_q)$ and $|\lambda-q|<\eps_0(F_q)$. Extracting a finite subcover of the compact
interval $[-R,R]$ from the open intervals $(q-\eps_0(F_q),\,q+\eps_0(F_q))$ and taking $\eps_1(F)$ to
be the minimum of the corresponding thresholds $\eps_0(F_q)$ gives the claim.
\end{proof}

For each fixed $\eps>0$, $(x,y,t)\mapsto(x,\sqrt\eps\,y,t/\sqrt\eps)$ takes the classical system
with polynomial $G$ to~\eqref{eq:system} with $F=\sqrt\eps\,G$ and $\lambda=0$. For arbitrary
degree let $N_F$ count isolated periodic orbits, since reversible centers can occur. The
rescaling is bijective on the polynomials of degree $n$, so $L(n)=\sup_{\deg F=n}N_F(\eps,0)$.
For all small $\eps>0$, the fixed $F_*$ in Remark~\ref{rem:sharp} yields two classical limit
cycles for $G=F_*/\sqrt\eps$. This motivates the following conjecture and open problem.

\begin{conjecture}\label{conj:attainment}
For every $n\ge6$: (i) the classical upper bound $L(n)$ is finite; and (ii) $L(n)$ is attained in
the slow--fast limit by a single fixed polynomial, that is, there are a real polynomial $F_n$ of
degree $n$ and $\eps_0>0$ such that $N_{F_n}(\eps,0)=L(n)$ for every $0<\eps<\eps_0$.
\end{conjecture}

\begin{openproblem}\label{prob:fixed-F}
For each integer $n\ge1$, define
\[
B(n)=\sup_{\deg F=n}\ \inf_{\delta>0}\ \sup_{0<\eps<\delta,\ |\lambda|<\delta}N_F(\eps,\lambda).
\]
For $n\ge6$, determine (i) whether $B(n)$ is finite, (ii) the exact value of $B(n)$ (in particular,
whether $B(n)=L(n)$), and (iii) whether $B(n)$ is attained by a single fixed polynomial $F_n$ for all
sufficiently small $\eps>0$.
\end{openproblem}

By Theorem~\ref{thm:main} and Remark~\ref{rem:sharp}, $B(5)=2$ with a threshold
depending on $F$. Translation to the equilibrium and the rescaling above give $B(n)\le L(n)$.
If planar polynomial vector fields of a given degree $n$ admit a uniform geometric upper bound
$H(n)<\infty$ (as claimed in a separate contemporaneous preprint, distinct
from~\cite{OpenAI2026Quintic}), then $B(n)\le L(n)\le H(n)<\infty$. Even conditionally on finiteness,
determining the exact value of $B(n)$ and resolving the attainment question in
Conjecture~\ref{conj:attainment}(ii) remain open.

Our proof proceeds by ordering the Hausdorff limits of nested periodic orbits as
$(\eps,\lambda)\to(0^+,0)$. Whenever a limiting slow--fast cycle has nonzero slow divergence
integral, this integral determines the stability of nearby limit cycles. Near a cycle whose slow
divergence integral vanishes, we instead compare the derivatives of the forward and backward
transition maps across the entire shooting interval. Periodic orbits shrinking to the equilibrium
are handled by a strip uniqueness criterion in the nondegenerate case, and by a three-scale
analysis at a degenerate minimum. Combining these stability and monotonicity properties yields
sign constraints on the return displacement that preclude three consecutive limit cycles.

The paper is organized as follows. Sections~\ref{sec:prelim} and~\ref{sec:limits} establish the
preliminary geometry and the partial order on limit periodic sets, while
Sections~\ref{sec:stability}--\ref{sec:neutral} develop the stability and transition-map
estimates. The configurations of the critical curve are then classified in
Sections~\ref{sec:minimum} and~\ref{sec:mono}, and Theorem~\ref{thm:main} is proved in
Section~\ref{sec:final}. Technical details on the three-scale estimates, the algebraic
comparisons, and the computer-assisted inequality are deferred to
Appendices~\hyperref[app:shrinking]{A}--\hyperref[app:cert]{C}.
\section{Preliminaries}\label{sec:prelim}

\subsection{Slow--fast structure}\label{sec:sf}

The layer equation at $\eps=0$ has critical curve $\mathcal C_0=\{y=F(x)\}$. Normal attraction or
repulsion corresponds to $F'>0$ or $F'<0$, respectively. The zeros of $F'$ are contact points. A
quadratic fold has $F''\ne0$ and, if its abscissa is nonzero, is a jump point in the sense of
\cite[Definition~2]{DeMaesschalckDumortierRoussarie2011}. In slow time
$\tau=\eps t$ at $\lambda=0$, the reduced equation is
\begin{equation}\label{eq:reduced}
F'(x)\,\frac{dx}{d\tau}=-x,\qquad\text{equivalently}\qquad \frac{dy}{d\tau}=-x\ \text{ on }\mathcal C_0 .
\end{equation}
The reduced flow descends on the right and ascends on the left. When $F'(0)=0$,
the origin is the \emph{turning point}, a singular contact point where the
limiting equilibrium lies. For $F''(0)\ne0$ it is a canard point
\cite[Section~3]{KS}. At a quadratic or quartic minimum, $d\tau/dx=-F'(x)/x$
extends continuously to $0$. The incoming slow arc on the right and the
outgoing one on the left meet there in finite slow time.

\begin{definition}\label{def:sfcycle}
A slow--fast cycle of system~\eqref{eq:system} is a closed curve $\Gamma\subset\mathbb R^2$ that is the
union of finitely many arcs of $\mathcal C_0$ (slow arcs) and finitely many horizontal segments
(fast segments), with the following properties.
\begin{enumerate}[label=(\roman*)]
\item On each slow arc the orientation of $\Gamma$ agrees with the reduced
flow~\eqref{eq:reduced}; slow arcs may contain contact points, including the turning point.
\item A fast segment at height $h$ that is traversed from left to right satisfies $h\ge F(x)$ at
all its points, and a fast segment traversed from right to left satisfies $h\le F(x)$ at all its
points.
\end{enumerate}
\end{definition}

A landing is regular if its endpoint on $\mathcal C_0$ is normally hyperbolic.
A slow--fast cycle is \emph{common} if its normally hyperbolic subarcs are
all attracting or all repelling, and is a \emph{canard cycle} otherwise
\cite[Definition~5]{DeMaesschalckDumortierRoussarie2011}. A canard through a
minimum is \emph{without head} when it has one slow arc and one fast
segment, and \emph{with head} otherwise~\cite{DumortierRoussarie1996,DMDRbook}.
The canard cycles of Section~\ref{sec:res} instead join equal-height folds.

\begin{definition}\label{def:sdi}
Let $\Gamma$ be a slow--fast cycle with slow arcs $\gamma_1,\dots,\gamma_k$, oriented as in
Definition~\ref{def:sfcycle}. Its slow divergence integral is
\begin{equation}\label{eq:sdi}
I(\Gamma)=\sum_{j=1}^k\int_{\gamma_j}\frac{F'(x)^2}{x}\,dx
=-\sum_{j=1}^k\int_{\gamma_j}F'(x(\tau))\,d\tau .
\end{equation}
\end{definition}

The equality follows from $dy=F'(x)\,dx=-x\,d\tau$. Write $Q(x)=F'(x)/x$ for $x\ne0$.
If $F'(0)=0$, the integrand extends through the turning point. In this case set $Q(0)=F''(0)$ and
$\cH(x)=\int_0^x Q(s)F'(s)\,ds$. For a regular moving root $F(e(h))=h$, $\frac{d}{dh}\cH(e(h))=Q(e(h))$. If $F'(0)\ne0$, the limiting slow arcs stay away from $x=0$ by
Lemma~\ref{lem:trap}. For left and right profiles $x_\pm(y)$ between heights $y_-$ and $y_+$, the
opposite orientations give
\begin{equation}\label{eq:sdi-height}
I(\Gamma)=\int_{y_-}^{y_+}\bigl[Q(x_-(y))-Q(x_+(y))\bigr]\,dy .
\end{equation}
Every nontrivial limit in Theorem~\ref{thm:limit-sets}(b) has this form. We call a slow--fast cycle $\Gamma$ \emph{neutral} if $I(\Gamma)=0$; in this paper the word
refers only to the vanishing of the slow divergence integral.

\subsection{Notation}\label{sec:notation}

The following notation is used throughout.
\begin{center}\small
\renewcommand{\arraystretch}{1.15}
\begin{tabular}{@{}ll@{}}
\toprule
$\Sigma$, $\cP$ & the upper half-line $\{x=\lambda,\ y>F(\lambda)\}$ and the return map on it\\
$\rho$ & the Floquet multiplier of a periodic orbit\\
$y_-<y_+$ & the heights at which a periodic orbit meets the line $x=\lambda$\\
$x_\pm(y)$, $W_\pm(y)$ & right and left half-orbits as graphs over $y$; $W_\pm=(x_\pm-\lambda)^2/2$\\
$\varrho_\pm$, $\varrho$ & $\max|x_\pm-\lambda|$ and $\varrho=\max(\varrho_+,\varrho_-)$, the amplitude\\
$Q$, $\cH$ & $Q(x)=F'(x)/x$; if $F'(0)=0$, also $\cH'=QF'$ and $\cH(0)=0$\\
$I(\Gamma)$ & the slow divergence integral of a slow--fast cycle $\Gamma$\\
$\Lambda_j$, $I_j$ & the $j$th family of slow--fast cycles in a chain, and $I_j=I\circ\Lambda_j$;\\
& a member shared by two consecutive families is their \emph{junction}\\
$\Pi_+$, $\Pi_-$ & forward and backward first-hit maps to a common section\\
$\theta$ & a coordinate on a transverse section that increases from inside to outside\\
$\varkappa_p$ & $|F''(p)|/2$ at a quadratic fold $p$\\
\bottomrule
\end{tabular}
\end{center}
Positive constants independent of the parameters are denoted by $c$, $K$, $K_1,\dots$; their values may change
line to line.

\subsection{Periodic orbits as pairs of graphs}

\begin{lemma}\label{lem:half-orbits}
Every periodic orbit of system~\eqref{eq:system} meets the line $x=\lambda$ at exactly two points, at
heights $y_-<F(\lambda)<y_+$. Its parts in $\{x>\lambda\}$ and $\{x<\lambda\}$ are graphs
$x=x_+(y)$ and $x=x_-(y)$ over $(y_-,y_+)$, and
\begin{equation}\label{eq:half}
\eps W_\pm'(y)=F(x_\pm(y))-y,\qquad W_\pm(y_-)=W_\pm(y_+)=0,\qquad
\int_{y_-}^{y_+}\bigl|F(x_\pm(y))-y\bigr|\,dy=\eps\varrho_\pm^2 .
\end{equation}
Each $W_\pm$ is unimodal: it increases up to its unique maximum $\varrho_\pm^2/2$ and then
decreases.
\end{lemma}

\begin{proof}
Every periodic orbit surrounds the unique equilibrium. Each half-line on $x=\lambda$ has a fixed
crossing direction, so the orbit crosses it once. The height is strictly monotone on either half,
and division of $\dot W_\pm$ by $\dot y$ gives~\eqref{eq:half}. At a critical point of $W_\pm$,
$\eps W_\pm''=-1$. Thus there is one maximum and the total variation equals $\varrho_\pm^2$,
giving the integral in~\eqref{eq:half}.
\end{proof}

\begin{proposition}\label{prop:compact}
There is a compact set $\mathcal K\subset\mathbb R^2$ containing all periodic orbits
of system~\eqref{eq:system} for all sufficiently small $\eps>0$ and $|\lambda|$. Moreover, along any
sequence of periodic orbits with $(\eps,\lambda)\to(0^+,0)$, the amplitude $\varrho$ tends to zero
if and only if $y_+-y_-$ tends to zero.
\end{proposition}

\begin{proof}
Uniformly for small $|\lambda|$, choose $K>|F(\lambda)|$ with $F\ge-K$
on $x\ge\lambda$ and $F\le K$ on $x\le\lambda$. Integrating
\eqref{eq:half} on the outer parts of the height range gives
\[
 (-K-y_-)_+\le\sqrt\eps\,\varrho_+,\qquad
 (y_+-K)_+\le\sqrt\eps\,\varrho_- .
\]
For example, when $y_-<-K$, the integral of $-K-y$ up to $-K$ is $(-K-y_-)^2/2\le\eps W_+(-K)$.
At the horizontal extremes $y=F(x)$, so quintic growth bounds $\varrho_\pm$ by $C(1+H^{1/5})$,
where $H=\max(1,|y_-|,|y_+|)$. Thus $H\le C(1+\sqrt\eps H^{1/5})$. If $\varrho\to0$, apply the
same height bounds with the minimum and maximum of $F$ on the orbit. Conversely, both horizontal
crossing directions give $y_-\le F(x)\le y_+$ across the projection. Since $\lambda\to0$ and $F$
is nonconstant on every interval, vanishing height span forces vanishing amplitude.
\end{proof}

Center the coordinates at the equilibrium, $z=x-\lambda$ and $Y=y-F(\lambda)$, and split
\[
F(\lambda\pm\sqrt{2W})-F(\lambda)=\Fe(W)\pm\Fo(W),
\]
so that $\Fe$ and $\Fo$ are the even and the odd part of $z\mapsto F(\lambda+z)-F(\lambda)$,
written as functions of $W=z^2/2$. Explicitly,
\begin{equation}\label{eq:odd-part}
\Fo(W)=z\Bigl[F'(\lambda)+\tfrac16F'''(\lambda)z^2+a_5z^4\Bigr],\qquad z=\sqrt{2W}.
\end{equation}

\begin{lemma}\label{lem:odd}
Let $D=W_+-W_-$ and
\[
A(Y)=\int_0^1\Fe'(W_-+\vartheta D)\,d\vartheta,\qquad
\omega(Y)=\exp\Bigl(-\eps^{-1}\int_0^YA(v)\,dv\Bigr).
\]
Every periodic orbit satisfies, with $W_\pm(Y)$ denoting the half-orbit profiles evaluated at
$y=Y+F(\lambda)$,
\begin{equation}\label{eq:odd}
\int_{y_--F(\lambda)}^{y_+-F(\lambda)}\omega(Y)\bigl[\Fo(W_+(Y))+\Fo(W_-(Y))\bigr]\,dY=0 .
\end{equation}
Consequently, if $\Fo$ has a constant strict sign on $(0,r_0^2/2]$, then no periodic orbit has
amplitude $\varrho\le r_0$.
\end{lemma}

\begin{proof}
Subtracting the two scalar equations gives $\eps(\omega D)'=\omega[\Fo(W_+)+\Fo(W_-)]$. Both
endpoint values of $D$ vanish, giving~\eqref{eq:odd}. If $\varrho\le r_0$ the integrand has the
same strict sign on the interior, a contradiction.
\end{proof}

\begin{lemma}\label{lem:isolated}
Every periodic orbit of system~\eqref{eq:system} is isolated.
\end{lemma}

\begin{proof}
By~\eqref{eq:odd-part}, the odd part has constant strict sign on a sufficiently small positive
interval, so Lemma~\ref{lem:odd} excludes a center. The determinant at the equilibrium $E$ is
$\eps>0$. Thus $E$ is a hyperbolic repeller or attractor, or a weak focus. Choose the time
direction in which it repels and a small neighborhood $U$ whose punctured interior is left
by every orbit in forward time and cannot be reentered.

If $\Gamma$ is a periodic orbit, the flow inside it is bounded. Every orbit in
$\Int\Gamma\setminus\{E\}$ has an $\omega$-limit set disjoint from a smaller neighborhood of $E$.
By Poincar\'e--Bendixson this limit is a periodic orbit, which crosses $\Sigma$ transversally
by Lemma~\ref{lem:half-orbits}. A flow box therefore gives a first return at every point of the
open segment of $\Sigma$ between $E$ and $\Gamma$. Transversality makes the return map analytic
there and across $\Gamma\cap\Sigma$. If $\Gamma$ were nonisolated, the identity theorem would
make this map the identity throughout the segment, giving periodic orbits through $U$.
\end{proof}

\subsection{The return map and consecutive limit cycles}

\begin{lemma}\label{lem:annulus}
Let $\Gamma\subset\Int\Gamma'$ be periodic orbits of system~\eqref{eq:system}. The return map $\cP$ is
defined and real analytic on the closed segment of $\Sigma$ between $\Gamma$ and $\Gamma'$, and
$\cP'>0$ there. In particular, only finitely many periodic orbits lie between $\Gamma$ and
$\Gamma'$.
\end{lemma}

\begin{proof}
The closed annulus is invariant and has no equilibrium. Poincar\'e--Bendixson and
Lemma~\ref{lem:half-orbits} give a transverse return to $\Sigma$ for every point of the segment.
The map is analytic, and the backward flow gives its local inverse. The maps at transverse
crossings preserve the outward order, so $\cP'>0$. The zeros of $\cP-\mathrm{id}$ on the compact
segment are finite by analyticity and Lemma~\ref{lem:isolated}.
\end{proof}

Orient $\Sigma$ outward by increasing $y$. Two periodic orbits are
\emph{consecutive} if the open annulus they bound contains no
periodic orbit. The \emph{sign pattern}
$(\sigma^{\mathrm{in}},\sigma^{\mathrm{out}})$ records the signs of $\cP-\mathrm{id}$ immediately inside
and outside a cycle and is well defined by isolation. Attractors have $(+,-)$, repellers $(-,+)$,
and semistable cycles $(+,+)$ or $(-,-)$. Outward section holonomy conjugates the return maps and
preserves the pattern.

\begin{lemma}\label{lem:consecutive}
If $\Gamma\subset\Int\Gamma'$ are consecutive periodic orbits, then
$\sigma^{\mathrm{out}}(\Gamma)=\sigma^{\mathrm{in}}(\Gamma')$. In particular, if $\Gamma$ and
$\Gamma'$ are both hyperbolic, one of them is attracting and the other is repelling.
\end{lemma}

\begin{proof}
The displacement $\cP-\mathrm{id}$ is continuous and has no zero on the open segment of $\Sigma$
between $\Gamma$ and $\Gamma'$, so it has a constant sign there.
\end{proof}
\section{Limit periodic sets}\label{sec:limits}

Proposition~\ref{prop:compact} and Blaschke's selection theorem give Hausdorff subsequential
limits. The \emph{profile} of an orbit or slow--fast cycle is $(y_-,y_+,x_-,x_+)$, with the
limiting functions defined almost everywhere on the intervening interval.

\begin{definition}\label{def:order}
For two slow--fast cycles $\Gamma$ and $\Gamma'$ with profiles $(y_-,y_+,x_-,x_+)$ and
$(y'_-,y'_+,x'_-,x'_+)$ we write $\Gamma\preceq\Gamma'$ if $y'_-\le y_-$, $y_+\le y'_+$ and
\[
x'_-(y)\le x_-(y)\le x_+(y)\le x'_+(y)\qquad\text{for a.e. }y\in(y_-,y_+).
\]
We also set $\{(0,0)\}\preceq\{(0,0)\}$ and $\{(0,0)\}\preceq\Gamma$ for every slow--fast cycle
$\Gamma$.
\end{definition}

\begin{theorem}\label{thm:limit-sets}
Let $\Gamma_n$ be periodic orbits of system~\eqref{eq:system} at parameters $(\eps_n,\lambda_n)\to(0^+,0)$.
After passing to a subsequence, $\Gamma_n$ converges in the Hausdorff distance to a compact set
$\Gamma_*$, and the following hold.
\begin{enumerate}[label=(\alph*)]
\item If $y_+-y_-\to0$ along the subsequence, then $\Gamma_*=\{(0,0)\}$.
\item Otherwise the intercepts converge to limits $y_-^*<y_+^*$, the half-orbits converge pointwise
to functions $x_\pm^*$ on $(y_-^*,y_+^*)$ with $F(x^*_\pm(y))=y$ for a.e.\ $y$, and $\Gamma_*$ is a
slow--fast cycle with profile $(y_-^*,y_+^*,x_-^*,x_+^*)$. Each of $x_+^*$ and $x_-^*$ follows
finitely many inverse branches of $F$, switching between them by finitely many jumps.
\item If moreover $\Gamma'_n$ are periodic orbits at the same parameters with
$\Gamma_n\subset\overline{\Int\Gamma'_n}$, converging to $\Gamma'_*$, then
$\Gamma_*\preceq\Gamma'_*$.
\end{enumerate}
A slow--fast cycle is determined by its profile, so $\preceq$ is a partial order on the set of
limits.
\end{theorem}

\begin{proof}
\emph{Step 1. Compactness and half-profiles.} By Proposition~\ref{prop:compact} and Blaschke's
theorem take $\Gamma_n\to\Gamma_*$ and $y_\pm^n\to y_\pm^*$. If $y_-^*=y_+^*$, the amplitudes
tend to zero and the limit is the origin. Otherwise extend $W_\pm^n$ by zero. They have uniformly
bounded variation by Lemma~\ref{lem:half-orbits}. Helly's theorem and convergence of their
maximizing heights along a subsequence give pointwise unimodal limits $W_\pm^*$.
For $x_\pm^*=\pm\sqrt{2W_\pm^*}$, Fatou's lemma and~\eqref{eq:half} give
\[
\int_{y_-^*}^{y_+^*}|F(x_\pm^*)-y|\,dy\le\liminf_n\eps_n(\varrho_\pm^n)^2=0.
\]
Between the limiting maxima, critical heights and $F(0)=0$, each profile is monotone.
At each continuity point it equals a regular inverse branch, and the branches are locally
separated. At a discontinuity the one-sided limits lie on distinct branches, whose indices
in their increasing order can change only in the direction of monotonicity. Thus there are
finitely many jumps. Unimodality and positivity almost everywhere give $W_\pm^*>0$ throughout
the interior height interval, with positive one-sided traces there.

\emph{Step 2. Closed horizontal slices.} If trajectories cross $[\xi_1,\xi_2]$ rightward
(resp.\ leftward) in arbitrarily thin height windows about $h$, the fast equation gives
$F\le h$ (resp.\ $F\ge h$) on that interval. In particular, both crossings are impossible
on an interval of positive length because $F$ is nonconstant.

Take Hausdorff limits of the two closed half-orbits. At each height the slice of a limiting
half is the closed interval between its two profile traces, with the missing trace at an
extremal height set equal to zero. Indeed, pointwise convergence at heights tending to either
side, or convergence of the intercept at a missing side, and continuity fill this interval.
An abscissa strictly beyond both traces would force crossings in both directions on a fixed
subinterval beyond the traces in arbitrarily thin windows, contradicting the observation.

\emph{Step 3. Orientation.} The branch pieces lie on $\mathcal C_0$ and inherit the reduced
orientation, descending on the right and ascending on the left. Each interior horizontal
segment is crossed in its profile direction in arbitrarily thin windows, so it satisfies
the jump condition by Step~2. The same argument on the two halves gives the rightward upper
and leftward lower segments. This proves (b).

\emph{Step 4. Order.} The height and half-profile inequalities follow from the horizontal
slices of the nested closed interiors and pass to the limit. If the outer orbit shrinks,
so does the inner one. Mutual comparison gives equal extreme heights and equal profiles
almost everywhere, hence equal one-sided traces and equal slices by Step~2. This proves
antisymmetry, and reflexivity and transitivity follow from Definition~\ref{def:order}.
\end{proof}

We call the slow--fast cycles in Theorem~\ref{thm:limit-sets}(b) \emph{admissible} and
condition~(ii) of Definition~\ref{def:sfcycle} the \emph{jump condition}.

\begin{lemma}\label{lem:shooting-annulus}
Let $\Gamma_n\subset\overline{\Int\Gamma'_n}$ be periodic orbits at the same parameters
$(\eps_n,\lambda_n)\to(0^+,0)$, both converging to a slow--fast cycle $\Gamma_*$.
Then the closed annuli
\[
 \mathcal A_n=\overline{\Int\Gamma'_n}\setminus\Int\Gamma_n
\]
converge to $\Gamma_*$ in the Hausdorff distance. Consequently, for every neighborhood of
$\Gamma_*$ and all large $n$, every trajectory starting on a transverse segment in $\mathcal A_n$
that joins the two orbits stays in this neighborhood in both time directions. Moreover, on a
compact interval of interior heights on which one half of $\Gamma_*$ is a single continuous
branch, the corresponding half-passages stay uniformly in any fixed tube about this branch, and
they cross transversally the horizontal sections through regular slow arcs away from $x=0$ and
the vertical sections through interior points of fast segments where $y\ne F(x)$.
\end{lemma}

\begin{proof}
The annuli lie in a fixed rectangle by Proposition~\ref{prop:compact}. At an interior limiting
height, an accumulation point of their slices lies between two same-side boundary crossings,
whose limits belong to the same closed slice of $\Gamma_*$. At either extremal height we
use the two outer-boundary crossings instead, since the full limiting slice there is one
closed interval. Thus $\limsup\mathcal A_n\subset\Gamma_*$, and $\Gamma_n\subset\mathcal A_n$
gives Hausdorff convergence.

Invariance confines the whole shooting segment in both time directions. Where the limiting
half is a single continuous branch, compactness gives uniform tube containment. On each
indicated local section the normal velocity has a strict fixed sign. The boundary orbits
cross it once locally, and the segment joining their crossings lies in the annulus. Each
periodic orbit in the annulus has the same crossing by nesting. Poincar\'e--Bendixson gives
periodic $\alpha$- and $\omega$-limits in the annulus, and a flow box at their crossings
gives the asserted passages for every point of the shooting segment. With coincident
boundaries the assertion applies to their common orbit.
\end{proof}
\section{Stability near a slow--fast cycle}\label{sec:stability}

Let $\Gamma$ be a periodic orbit of system~\eqref{eq:system} with period $T$ and Floquet multiplier
$\rho$. Since the divergence of system~\eqref{eq:system} is $-F'(x)$, Liouville's formula gives
$\log\rho=-\int_0^TF'(x(t))\,dt$. Write
\[
Q_\lambda(x)=\frac{F'(x)-F'(\lambda)}{x-\lambda}\qquad(Q_\lambda(\lambda)=F''(\lambda)).
\]
Splitting off $F'(\lambda)$ and using $\eps\,dt=-dy/(x-\lambda)$ on each half-orbit, we obtain the
exact identity
\begin{equation}\label{eq:exact-mult}
\eps\log\rho=\int_{y_-}^{y_+}\bigl[Q_\lambda(x_-(y))-Q_\lambda(x_+(y))\bigr]\,dy-\eps F'(\lambda)\,T .
\end{equation}
Since $Q_\lambda(x)=\int_0^1F''(\lambda+\vartheta(x-\lambda))\,d\vartheta$, the functions $Q_\lambda$
converge, as $\lambda\to0$ and uniformly on compact sets, to $\int_0^1F''(\vartheta x)\,d\vartheta$,
which equals $Q$ when $F'(0)=0$.

\begin{lemma}\label{lem:trap}
Suppose that $F$ is strictly monotone on a neighborhood of $x=0$ (possibly with $F'(0)=0$). There
are $d_0,h_0>0$ such that, for all sufficiently small $\eps$ and $|\lambda|$, every periodic orbit is
disjoint from the rectangle $\mathcal R_0=(-d_0/2,d_0/2)\times(-h_0,h_0)$ and satisfies
$y_-\le-h_0<h_0\le y_+$.
\end{lemma}

\begin{proof}
In the coordinates $u=x-\lambda$, $v=y-F(\lambda)$, put $G(u)=F(u+\lambda)-F(\lambda)$ and use
the time $s=\sigma t$, where $\sigma=1$ if $G$ increases near $0$ and $\sigma=-1$ otherwise. Then
$\frac{d}{ds}(\eps u^2+v^2)=-2\sigma\eps uG(u)<0$ for small $u\ne0$. Fix a small $d_0>0$ and choose $b>0$
smaller than $|G(\pm d_0)|$ for all small $|\lambda|$. The vector field points inward on the walls
$u=\pm d_0$ for $|v|\le b$, so $\{|u|\le d_0,\ \eps u^2+v^2\le b^2\}$ is invariant in the chosen time
direction, and no periodic orbit meets it. For small $\eps$ this set contains
$\mathcal R_0$ with $h_0=b/2$. Since $E$ is in the rectangle, the two intercepts lie on its opposite sides.
\end{proof}

At the origin a quintic is either strictly monotone or has a nondegenerate or quartic extremum.

\begin{proposition}\label{prop:sdi}
Let $\Gamma_n\to\Gamma_*$ be as in Theorem~\ref{thm:limit-sets}(b), with Floquet multipliers
$\rho_n$. Then $\eps_n\log\rho_n\to I(\Gamma_*)$. Consequently, if $I(\Gamma_*)<0$
(resp.\ $I(\Gamma_*)>0$), then $\Gamma_n$ is a hyperbolic attracting (resp.\ repelling) limit
cycle for all large $n$.
\end{proposition}

\begin{proof}
\emph{Extremum.} We first show $\eps|\lambda|T\to0$, taking $\lambda\ne0$ since the product
vanishes at $\lambda=0$. Put $z=x-\lambda$, $v=\dot x$ and
$\mathcal V=(\eps z^2+v^2)/2$. Then $\dot{\mathcal V}=-F'(x)v^2$, and $F'$ has a constant sign on
$|z|\le|\lambda|/2$. Hence the ellipse $\{\mathcal V\le\eps\lambda^2/8\}$, which lies in
$\{|z|\le|\lambda|/2\}$, is invariant in the time direction in which $\mathcal V$ decreases. A
periodic orbit meeting it would stay in it for a whole period, along which
$\dot{\mathcal V}=-F'(x)v^2$ has a constant sign and does not vanish identically; this is
impossible. Hence no periodic orbit meets the ellipse. Outside the ellipse and on
$|z|<|\lambda|/4$ we have $|v|\ge\frac{\sqrt3}{4}\sqrt\eps|\lambda|$; since each half-orbit crosses
this strip at most twice, the time spent there is $O(\eps^{-1/2})$, which contributes
$O(\sqrt\eps|\lambda|)$ to $\eps|\lambda|T$.

Let $h>0$ and choose $r>0$ such that $|F(x)|<h/2$ for $|z|\le r$ and all small $|\lambda|$. On
$|\lambda|/4\le|z|\le r$, the part of the orbit with $|y|\le h$ contributes at most $16h$ to
$\eps|\lambda|T$, because $\eps\,dt=|dy|/|z|\le4|dy|/|\lambda|$. The part with $|y|>h$ satisfies
$|v|\ge h/2$ and takes time $O(r/h)$. Outside $|z|\le r$, the relation $\eps\,dt=|dy|/|z|$ bounds
the contribution to $\eps|\lambda|T$ by $2|\lambda|(y_+-y_-)/r$. Let $(\eps,\lambda)\to(0^+,0)$
and then $h\downarrow0$. As $F'(\lambda)=O(|\lambda|)$ and
$Q_\lambda\to Q$ uniformly on compact sets,~\eqref{eq:exact-mult} reduces to the half-profile
integral with $Q$ and an $o(1)$ error. Dominated convergence gives~\eqref{eq:sdi-height}.

\emph{Locally monotone case.} Use the unsplit integral of $F'(x_\pm)/(x_\pm-\lambda)$ in
\eqref{eq:exact-mult}. The limiting slow arcs stay away from $x=0$: they avoid $\mathcal R_0$
by Lemma~\ref{lem:trap} and satisfy $F(x)=y$. Choose $a>0$ smaller than half their distance from
$x=0$, with $2a<d_0/2$ and $|F(x)|<h_0/2$ for $|x|\le2a$, and take $|\lambda|<a$. On
$|x-\lambda|\ge a$ dominated convergence gives~\eqref{eq:sdi-height}. On $|x-\lambda|<a$, the
orbit avoids $\mathcal R_0$, so $|\dot x|\ge h_0/2$. Unimodality bounds the number of traversals,
and their total time is $O(a)$. Their contribution to $\eps\log\rho$ is $O(\eps a)$.
\end{proof}

For an admissible limit $\Gamma_*$ whose contact points are all regular, these points are nilpotent
and of finite order, and its normally hyperbolic slow arcs contain no equilibrium of the reduced
flow (by Lemma~\ref{lem:trap} when $F'(0)\ne0$). In this case the divergence estimate
of~\cite[Theorem~3]{DeMaesschalckDumortierRoussarie2011} also gives
$\eps_n\log\rho_n\to I(\Gamma_*)$; the admissible cycles considered here may in addition pass
through the turning point.

\begin{corollary}\label{cor:one-type}
If every sequence of periodic orbits as $(\eps,\lambda)\to(0^+,0)$ has a subsequence of
hyperbolic orbits of the same prescribed stability, then $N_F\le1$ for all small parameters.
In particular, it suffices that every admissible cycle satisfy $I<0$ and every sequence
shrinking to the origin be eventually hyperbolic and attracting.
\end{corollary}

\begin{proof}
Otherwise choose two consecutive orbits at each parameter pair and pass to subsequences
where both have the prescribed stability, contrary to Lemma~\ref{lem:consecutive}.
The sufficient condition gives such subsequences by Theorem~\ref{thm:limit-sets} and
Proposition~\ref{prop:sdi}. It holds for the nontrivial limits whenever
$Q(x_-(y))<Q(x_+(y))$ almost everywhere by~\eqref{eq:sdi-height}.
\end{proof}
\section{Limit cycles shrinking to the equilibrium}\label{sec:shrinking}

By Lemma~\ref{lem:trap}, no periodic orbit shrinks to the equilibrium when $F$ is strictly
monotone near $x=0$. It remains to consider a local extremum at $x=0$, which we normalize to a
minimum (Section~\ref{sec:min-setup}). Then
$F(x)=a_2x^2+a_3x^3+a_4x^4+a_5x^5$ with $a_2,a_5>0$ at a nondegenerate minimum, and
$F(x)=a_4x^4+a_5x^5$ with $a_4,a_5>0$ at a quartic minimum.

\begin{proposition}\label{prop:shrinking}
Along every sequence of periodic orbits that shrinks to $(0,0)$ as $(\eps,\lambda)\to(0^+,0)$, the
orbits are eventually hyperbolic. They are repelling if $a_2>0$ and $a_3<0$, and attracting
otherwise.
Moreover, if $a_3\ne0$, there is a fixed neighborhood of the origin that contains no periodic orbit
when $a_3\lambda\ge0$ and at most one periodic orbit when $a_3\lambda<0$.
\end{proposition}

\begin{proof}[Proof for $a_3\ne0$]
Write $z=x-\lambda$. Choose $r>0$ so small that $F'''$ has the sign of $a_3$ on $[-3r,3r]$ and
$a_5r^2<|a_3|/4$. For all sufficiently small $|\lambda|<r/2$, $F'(\lambda)$ has the sign of
$\lambda$ when $\lambda\ne0$, and $|F'''(\lambda)/6-a_3|<|a_3|/4$.

Suppose first that $a_3\lambda\ge0$. By~\eqref{eq:odd-part},
\[
\frac{\Fo(W)}{z}=F'(\lambda)+z^2\Bigl(\frac{F'''(\lambda)}6+a_5z^2\Bigr),\qquad z=\sqrt{2W},
\]
and both terms have the sign of $a_3$ for $0<z\le r$, the second one strictly, also when
$\lambda=0$. Lemma~\ref{lem:odd} excludes every periodic orbit of amplitude at most $r$.

Suppose now that $a_3\lambda<0$. We use the uniqueness criterion of Zegeling and
Kooij~\cite[Lemma~2.3]{ZK}, a strip version of Zhang Zhifen's
theorem~\cite{Zhang,ZhangBook}. For a system $\dot z=\varphi(v)-\int_0^zh(s)\,ds$,
$\dot v=-g(z)$ on $I\times\mathbb R$, where $I$ is an open interval containing $0$, it assumes
$h,g\in C^1(I)$, $\varphi\in C^1(\mathbb R)$ with $\varphi'>0$, and $zg(z)>0$ and
$h(0)\,(h/g)'(z)<0$ for $z\in I\setminus\{0\}$; it then asserts that the system has at most one
limit cycle contained in the strip $I\times\mathbb R$, and that such a cycle is hyperbolic. In the
coordinates $z$ and $v=y-F(\lambda)$, system~\eqref{eq:system} has this form with $I=(-2r,2r)$,
$\varphi(v)=v$, $g(z)=\eps z$ and $h(z)=F'(\lambda+z)$. For $z\ne0$ we have $zg(z)=\eps z^2>0$, and
on each component of $I\setminus\{0\}$
\begin{equation}\label{eq:strip-quotient}
\eps\Bigl(\frac hg\Bigr)'(z)=\int_0^1\vartheta F'''(\lambda+\vartheta z)\,d\vartheta
-\frac{F'(\lambda)}{z^2}.
\end{equation}
Both terms have the strict sign of $a_3$, whereas $h(0)=F'(\lambda)$ has the opposite sign. By
Lemma~\ref{lem:isolated} every periodic orbit is a limit cycle, so the strip $|z|<2r$, and in
particular the fixed neighborhood $\{(x,y):|x|<r/2,\ |y|<r_1\}$ of the origin for a suitable
$r_1>0$, contains at most one periodic orbit, which is hyperbolic.

The equilibrium has trace $-F'(\lambda)$ and determinant $\eps>0$. It is a source for $a_3>0$
and a sink for $a_3<0$. If the hyperbolic cycle had the same stability, orient time so
both repel. An interior orbit outside a small repelling neighborhood would have an
$\omega$-limit disjoint from the equilibrium and the cycle. Poincar\'e--Bendixson would give a
second cycle in the strip, contradicting uniqueness. See also~\cite[Theorem~2.1]{ChengLi}.
\end{proof}

The degenerate cases $a_3=0$ and the quartic minimum are governed by the local polynomials
\[
F_{(1)}(x)=a_2x^2+a_4x^4+a_5x^5,\qquad F_{(2)}(x)=a_4x^4+a_5x^5,
\]
with $a_4\in\mathbb R$ in the first case and $a_4>0$ in the second. For $m\in\{1,2\}$ put
\[
\varrho=\max|x-\lambda|,\qquad \delta_m=\eps^{1/(4m-2)},\qquad R=\varrho/\delta_m .
\]
The scale $\delta_m$ balances the lowest term of $F_{(m)}$, of degree $2m$, against the restoring
term $\eps(x-\lambda)$.

\begin{proposition}\label{prop:shrinking-scales}
Let $m\in\{1,2\}$ and let $F=F_{(m)}$. Every sequence of periodic orbits with
$(\eps,\lambda)\to(0^+,0)$ and $\varrho\to0$ is eventually hyperbolic and attracting. If moreover
$R\to\infty$, then
\begin{align}
\log\rho&=-4a_2a_5\,\frac{\varrho^5}{\eps}\,(1+o(1)) &&(m=1),\label{eq:mult-1}\\
\log\rho&=-\frac{24}{7}\,a_4a_5\,\frac{\varrho^7}{\eps}\,(1+o(1)) &&(m=2).\label{eq:mult-2}
\end{align}
\end{proposition}

\begin{proof}
If attraction failed, a sequence with $\rho\ge1$ would have a subsequence on which $R$ tends to
$0$, a positive finite limit, or $\infty$. \ref{app:shrinking} gives $\rho<1$ in each
case and proves~\eqref{eq:mult-1} and~\eqref{eq:mult-2} with relative errors in the last regime.
This also completes the proof of Proposition~\ref{prop:shrinking}.
\end{proof}

The relative errors in \eqref{eq:mult-1} and \eqref{eq:mult-2}
determine the sign of $\log\rho$ even when the leading term
tends to zero. An additive $o(1)$ error alone would not
determine its sign in this regime.

\section{Limit cycles near a neutral slow--fast cycle}\label{sec:neutral}

We give a criterion for counting periodic orbits near a slow--fast cycle with $I=0$ and
derive the transition estimates used to verify it in Sections~\ref{sec:minimum} and~\ref{sec:mono}.

\subsection{The monotone ratio property}

\begin{definition}\label{def:N}
A slow--fast cycle $\Gamma_*$ has the \emph{monotone ratio property} if the following holds.
Let $\Gamma_n$ and $\Gamma'_n$ be periodic orbits at the same parameters
$(\eps_n,\lambda_n)\to(0^+,0)$, both converging to $\Gamma_*$. Suppose that for each $n$
either the orbits are strictly nested ($\Gamma_n\subset\Int\Gamma'_n$) or the two bounding
orbits coincide ($\Gamma_n=\Gamma'_n$). Then, for all large $n$, there are a transverse
section $S_n$ with a coordinate
$\theta$ increasing from inside to outside, an open interval $\mathcal J_n\subset S_n$ containing
the closed segment between $\Gamma_n$ and $\Gamma'_n$, a second transverse section $S'_n$ equipped with a $C^2$ coordinate (in our applications, the height coordinate $y$), and
$C^2$ maps $\Pi_\pm:\mathcal J_n\to S'_n$ such that
\begin{enumerate}[label=(\roman*)]
\item $\Pi_+$ and $\Pi_-$ are the forward and backward first-hit maps from $S_n$ to $S'_n$, chosen so that $\Pi_-^{-1}\circ\Pi_+$ is the local first-return map $\cP$ on $\mathcal J_n$ wherever it is defined; in particular, $\theta\in\mathcal J_n$ lies on a periodic orbit if and only if $\Pi_+(\theta)=\Pi_-(\theta)$;
\item $\Pi_+'\Pi_-'>0$ and $\bigl(\Pi_+'/\Pi_-'\bigr)'<0$ on $\mathcal J_n$, where derivatives are taken in the chosen coordinate on $S'_n$.
\end{enumerate}
\end{definition}

For related Rolle arguments
see~\cite[Remarks~1 and~4]{Huzak2019Mobius}.

\begin{theorem}\label{thm:neutral}
Let $\Gamma_*$ have the monotone ratio property, and let $\Gamma_n\subset\overline{\Int\Gamma'_n}$
be periodic orbits at parameters $(\eps_n,\lambda_n)\to(0^+,0)$, both converging to $\Gamma_*$.
Then, for all large $n$:
\begin{enumerate}[label=(\alph*)]
\item at most two periodic orbits lie in the closed annulus between $\Gamma_n$ and $\Gamma'_n$;
\item if $\Gamma_n\ne\Gamma'_n$, then $\Gamma_n$ is hyperbolic and repelling and $\Gamma'_n$ is
hyperbolic and attracting;
\item an orbit with unit multiplier has multiplicity exactly two and sign pattern $(-,-)$.
In the outward shooting coordinate its return map satisfies
$(\cP-\mathrm{id})''(\theta_0)=\varphi'(\theta_0)<0$, where $\varphi=\Pi_+'/\Pi_-'$.
The hyperbolic orbits have sign patterns $(-,+)$ or $(+,-)$.
\end{enumerate}
\end{theorem}

\begin{proof}
On $\mathcal J_n$ put $\varphi=\Pi_+'/\Pi_-'$ and $\cE=\sgn(\Pi_-')(\Pi_+-\Pi_-)$. The sign of
$\cE$ is that of $\cP-\mathrm{id}$ wherever $\cP=\Pi_-^{-1}\circ\Pi_+$ is defined. Because
$\cE'=|\Pi_-'|(\varphi-1)$ and $\varphi$ is strictly decreasing, $\cE$ has at most two zeros.
Every intervening cycle crosses the shooting segment by nesting. With two zeros the derivative
ratio is strictly above one at the first and strictly below one at the second, by Rolle's
theorem. At either zero the ratio is the return multiplier. At a fixed point $\theta_0$ with
$\cP'(\theta_0)=1$, differentiating $\Pi_+=\Pi_-\circ\cP$ twice gives
\[
(\cP-\mathrm{id})''(\theta_0)=
\frac{\Pi_+''(\theta_0)-\Pi_-''(\theta_0)}{\Pi_-'(\theta_0)}=\varphi'(\theta_0)<0 .
\]
This proves the multiplicity and the sign pattern on the open interval provided even for
coincident bounding orbits. The section holonomy to $\Sigma$ is an increasing $C^2$
diffeomorphism, so it preserves both conclusions.
\end{proof}

\subsection{Transition maps along normally hyperbolic branches}\label{sec:capture}

To verify Definition~\ref{def:N} we compute transition maps in the coordinates
\[
W=\frac{(x-\lambda)^2}{2},\qquad G(W)=F\bigl(\lambda\pm\sqrt{2W}\bigr),\qquad
G'(W)=\frac{F'(x)}{x-\lambda},
\]
in each half-plane $\pm(x-\lambda)>0$, and use the height $\sigma=y$ (ascending) or $\sigma=-y$
(descending) as the independent variable.

\begin{lemma}\label{lem:corridor}
Consider $\nu W_\sigma=\iota G(W)-\sigma$, $\iota=\pm1$, $W(\sigma_0;z)=z$ on a bounded height
interval in a region with $|G'|\le K_1$, $|G''|\le K_2\nu^{-M}$ and $|V|\le K_3/\nu$, where
$V=-W_\sigma$. Assume either $V(\sigma_0)\ge c>0$ or that $-K_1\le\iota G'\le-g<0$ on an initial
compact subarc of fixed height length $d=\sigma_1-\sigma_0>0$. After the initial subarc in the
second case, $V$ is bounded below by a positive constant. The maps to a fixed height and to a
fixed value $W_*$ after the positive speed bound takes effect have nonzero first derivatives and
polynomial bounds in $\nu^{-1}$ on both $|f'|$ and $|f''/f'|$. For the latter map
$\partial_z\sigma=E/V>0$, where $E=\partial_zW$, and for descending height $\partial_z y=-E/V<0$.
\end{lemma}

\begin{proof}
Put $a=\iota G'(W)$. Then
\[
\nu V_\sigma=1+aV,\qquad \nu E_\sigma=aE,\qquad (E/V)_\sigma=-E/(\nu V^2)
\]
where $V$ is nonzero. On the initial normally hyperbolic subarc, variation of constants gives
\[
 V(\sigma_1)\ge-K_3\nu^{-1}e^{-gd/\nu}+(1-e^{-K_1d/\nu})/K_1\ge(2K_1)^{-1}.
\]
Also $0<E\le e^{-g(\sigma-\sigma_0)/\nu}$ there, so the map to $\sigma_1$ has $0<E_1\le1$ and
$|\partial_z\log E_1|\le K_2\nu^{-M}/g$. With positive starting speed the first equation gives
$V\ge\min(c,1/(1+K_1))$, using the new lower bound for $c$ after an initial subarc. With the
variational solution normalized at that starting height, $E/V$ decreases, so $0<E\le K/\nu$.
Estimating $\partial_z\log E=(\iota/\nu)\int G''(W)E\,d\sigma$ gives $|\partial_z\log E|\le
K\nu^{-M-2}$. At $W(\sigma(z);z)=W_*$, implicit differentiation gives $\sigma'=E/V$. Its
logarithmic derivative has a polynomial bound by the same estimates, $V_z=-aE/\nu$ and the
displayed variational equations. Composing with the initial map gives the second case.
\end{proof}

For $D_f=f''/f'$, the identity $D_{f\circ g}=D_g+(D_f\circ g)g'$ shows that polynomial
bounds on $|f'|$ and $|D_f|$ persist under finite composition. If only one factor has
bounds of order $\nu^{-N}e^{c/\nu}$, the same exponential rate suffices for the composition.

\begin{lemma}\label{lem:composition}
On a closed shooting interval $\mathcal J_\nu$ (possibly a point), suppose the forward and
backward first-hit maps to a common transverse section factor as $\Pi_i=V_i\circ S_i\circ T_i$
for $i=+,-$. All factors are $C^2$ with nonzero derivative. On $\mathcal J_\nu$ assume
$T_i=\mathrm{id}+O_{C^2}(\nu)$. On $T_i(\mathcal J_\nu)$ assume $|S_i'|\le
K\nu^{-N}e^{-\chi_i/\nu}$, and on $S_i(T_i(\mathcal J_\nu))$ assume $|D_{V_i}|\le
K\nu^{-N}e^{c_i/\nu}$, with $0\le c_i<\chi_i$. On the respective entry images suppose
\[
 \nu D_{S_+}\le a_++o(1),\qquad
 \nu D_{S_-}\ge a_--o(1),
\]
where the bounded functions $a_i$ satisfy $a_+\circ T_+-a_-\circ T_-\le-d<0$.
All bounds are uniform. Then $\nu(\log(\Pi_+'/\Pi_-'))'\le-d+o(1)$ if the two maps
have the same section orientations.

In particular, if for every pair in Definition~\ref{def:N} these estimates hold on the closed
segment between the orbits, with common source and target sections of fixed crossing sign and
with one first hit of the target on each bounding orbit in each time direction, then they give
the monotone ratio property on an open neighborhood of that segment.
\end{lemma}

\begin{proof}
The chain rule gives $\nu D_{\Pi_i}=(\nu D_{S_i}\circ T_i)T_i'+o(1)$, since the continuation term
is $O(\nu^{1-2N}e^{-(\chi_i-c_i)/\nu})$ and $\nu D_{T_i}=O(\nu^2)$. Here $T_i'=1+O(\nu)>0$
preserves both inequalities. For a planar autonomous flow, the derivative of a transverse hit has
the sign of the ratio of the determinants of the vector field and the oriented section tangents
at source and target. Indeed, the flow Jacobian has positive determinant for either sign of the
flight time, and differentiation of the flight time only adds a multiple of the terminal vector
field. These signs agree for $\Pi_+$ and $\Pi_-$.

In the applications the target is the lower half-line at the equilibrium abscissa, equipped with the
height coordinate $y$. By Lemma~\ref{lem:half-orbits}, a periodic orbit hits it once, so
$\Pi_-^{-1}\circ\Pi_+$ is the local first-return map wherever it is defined, and equality of the two
hits is equivalent to a periodic orbit by uniqueness. Transversality continues the first hits across
each endpoint, and the strict derivative inequality persists there by continuity. The same argument
gives an open interval when the bounding orbits coincide.
\end{proof}

\begin{lemma}\label{lem:wedge}
Let $0$ be a nondegenerate minimum, and choose $d_0>0$ so that $Q>0$ on $[-2d_0,2d_0]$.
Suppose the lower part of a periodic orbit makes two complete crossings of fixed tubes
about the compact slow arcs over $d_0\le |x|\le2d_0$, from $|x|=2d_0$ to $|x|=d_0$ in forward
time on the right and in backward time on the left. Choose each tube's height width
less than one quarter of its arc's height span. Then $|\lambda|\le M\eps$ for all
sufficiently small parameters, uniformly for these passages.
\end{lemma}

\begin{proof}
Let $m(x)=F(x)-\eps(x-\lambda)/F'(x)$ on the two compact arcs. For $v=y-m(x)$ one has $\dot
v=-m'(x)v+O(\eps^2)$ and $m'=F'+O(\eps)$. In the indicated times the linear coefficient is
bounded above by a negative constant. The complete crossings span fixed positive height
intervals, so $|\dot y|\le C\eps$ gives residence times at least $c/\eps$. Thus $F(\pm d_0)-y(\pm
d_0)=\eps(\pm d_0-\lambda)/F'(\pm d_0)+O(\eps^2)$. The lower orbit is a graph, by the
monotonicity of each half-orbit next to the lower intercept. Its gap $p=F-y>0$ satisfies
$p'=F'-\eps(x-\lambda)/p$. For $p_0=\eps/Q$ one has $(p-p_0)'=\lambda Q-p_0'$ at a zero of
$p-p_0$, with $|p_0'|\le K\eps$. If $\lambda>M\eps$, the endpoint signs at $-d_0,d_0$ are
positive and negative, whereas the derivative at every zero is positive. For $\lambda<-M\eps$ all
three signs reverse. Both cases are impossible.
\end{proof}

We use Lemma~\ref{lem:wedge} on one bounding orbit by choosing the compact arcs in the lower slow
graph below all positive jump heights. Lemma~\ref{lem:shooting-annulus} gives the complete
crossings. The resulting parameter bound holds on the full shooting segment.

\begin{proposition}\label{prop:capture}
\textup{(a)} In either height direction let $W(h;h)=W_\xi=(\xi-\lambda)^2/2$ at a fixed
vertical cut $x=\xi$ with $|h-F(\xi)|\ge m>0$, and let $U(h)=W(b;h)$ for a fixed
height $b$ in the same half-plane. For $\nu W_y=G(W)-y$ the entry variations are
\begin{equation}\label{eq:entry-var}
\begin{aligned}
 W_h(y;h)&=\frac{h-G(W_\xi)}{\nu}
 \exp\Bigl(\frac1\nu\int_h^y G'(W(v;h))\,dv\Bigr),\\
 \nu D_U(h)&=\frac{\nu}{h-G(W_\xi)}-G'(W_\xi)+\int_h^b G''(W)W_h\,dy .
\end{aligned}
\end{equation}
If $|b-h|\le D$, $|G'|+|G''|\le K$ along the strip, and $dG'\le\mu$ for
$d=\sgn(b-h)$ and $\mu\ge0$, then
\begin{equation}\label{eq:receiving-strip}
 |U'|\le C\nu^{-1}e^{\mu D/\nu},\qquad
 |U''/U'|\le C\nu^{-2}e^{\mu D/\nu}.
\end{equation}

\textup{(b)} Suppose a jump from a fixed vertical section lands at a regular root $z(h)$
of $F(z)=h$ and follows a first compact normally hyperbolic slow arc of fixed positive height
length to a section $y=b$. The arc stays away from $x=\lambda$ and all critical points of $F$ and
is attracting in the chosen time direction. Suppose the jump has $|h-F(x)|\ge c>0$ up to a cut in
a normally hyperbolic tube at its entry. Then the map from that cut to $W(b)$ has
\[
 |U'|\le C\eps^{-1}e^{-\chi/\eps},\qquad
 \eps D_U=-Q(z(h))+o(1)
\]
in the height coordinate $h$. The preceding fast height map is $\mathrm{id}+O_{C^2}(\eps)$.

\textup{(c)} At a nondegenerate minimum $0$, with $|\lambda|\le M\eps$, consider a
continuation down $l$ or $r$ (the inverse branches of $F$ immediately to the left and right of $0$, respectively; see Section~\ref{sec:min-setup}) from a fixed height in a regular tube on the branch
adjacent to the minimum, below any stationary inflection. The map from $W$
to the height on $\{x=\lambda,\ y<F(\lambda)\}$ has negative derivative and
polynomial bounds on its derivative and logarithmic derivative.
\end{proposition}

\begin{proof}
Differentiating the scalar equation and the moving initial condition gives
\eqref{eq:entry-var}. The exponential there is at most $e^{\mu D/\nu}$ throughout the
strip, proving both estimates in~\eqref{eq:receiving-strip}.

On the regular entry tube $dG'\le-g<0$, so $|\eps W_h|\le C e^{-g|y-h|/\eps}$. In
$s=d(y-h)/\eps$, let $\bar W$ be the layer solution $\bar W_s=d(G_0(\bar W)-h)$ from $\xi^2/2$,
with $G_0$ the function $G$ at $\lambda=0$. On compact $s$-intervals $d\eps W_h\to-\bar W_s$
uniformly. The exponential bound and bounded $G''$ give, uniformly on compact entry intervals,
\[
 \int_h^b G''(W)W_h\,dy\longrightarrow
 -\int_0^\infty G_0''(\bar W)\bar W_s\,ds
 =G_0'(\xi^2/2)-Q(z(h)).
\]
Substitution in~\eqref{eq:entry-var} gives (b), with any $\chi$ below the product
of $g$ and the minimum height length in the tube. Before the cut the right-hand side
of $dy/dx=-\eps(x-\lambda)/(y-F(x))$ is $O_{C^2}(\eps)$ by the clearance hypothesis.

For (c) use $\sigma=-y$. On the initial compact arc $G'\ge g>0$ and $|V|\le C/\eps$, so
Lemma~\ref{lem:corridor} applies down to $|x-\lambda|=L_0\sqrt\eps$ for fixed $L_0>0$. Indeed, on
the descent $G'\ge c>0$ and $G''=((x-\lambda)F''(x)-F'(x))/(x-\lambda)^3=O(\eps^{-1/2})$. The
height map has $\partial_z\sigma=E/V>0$ and $\partial_z y=-E/V<0$.
Since $G'=2a_2+O(d_1+\sqrt\eps)$ on $L_0\sqrt\eps\le|x-\lambda|\le d_1$, variation of constants in
$\eps V_\sigma=1-G'V$ gives, at $|x-\lambda|=L_0\sqrt\eps$,
\[
|V-(2a_2)^{-1}|\le C_{d_1}e^{-cd_1^2/\eps}+C(d_1+\sqrt\eps).
\]
Letting first $\eps\to0$ and then $d_1\to0$, and using $y=F(x)-\eps V$, we see that the rescaled
entry height tends to $a_2L_0^2-(2a_2)^{-1}$. In the variables
$X=(x-\lambda)/\sqrt\eps$, $Y=(y-F(\lambda))/\eps$ and the time $\sqrt\eps\,t$, the system converges
in $C^2$ to $X'=Y-a_2X^2$, $Y'=-X$, whose invariant parabola $Y=a_2X^2-1/(2a_2)$ starting from
$X=\pm L_0$ crosses $X=0$ transversally, with $X'=-(2a_2)^{-1}$ there; this is the rescaling chart
of the canard point in~\cite[Section~3]{KS}. Smooth dependence on initial conditions and the
implicit function theorem give $C^2$ convergence of the hitting maps on a fixed neighborhood of
the rescaled entry height, with first derivatives bounded away from zero. In the original height
the normalized second derivative is $O(\eps^{-1})$. The local height map is increasing there by
transverse flow orientation, so its composition with the first map gives (c).
\end{proof}

\subsection{Exit at a quadratic fold}\label{sec:airy}

We use the classical quadratic fold asymptotic~\cite{MishchenkoRozov} uniformly in $\lambda$.

\begin{proposition}\label{prop:fold-exit}
Let $p\ne0$ be a quadratic fold of $F$, approached along a normally attracting slow arc of fixed
positive length in the chosen time direction, and let $\zeta$ be the smallest positive number with
$\Ai(-\zeta)=0$. At a fixed outgoing vertical section near $p$, the height of the trajectory is
\[
y_{\mathrm{out}}=F(p)+\iota\,\zeta\Bigl(\frac{|p-\lambda|^2}{\varkappa_p}\Bigr)^{1/3}\eps^{2/3}+o(\eps^{2/3}),
\]
where $\iota=+1$ at a maximum of $F$ and $\iota=-1$ at a minimum. The remainder is uniform on
compact families of incoming orbits. The same formula holds for a repelling arc followed in
backward time.
\end{proposition}

\begin{proof}
Put $a=\varkappa_p$ and $b=|p-\lambda|$. With reflections adapted to the time
direction the equations are $\dot\xi=\eta+a\xi^2+O(\xi^3)$,
$\dot\eta=\eps(b+O(\xi))$, with $\eta=\iota(y-F(p))$. The scaling
$\xi=(\eps b/a^2)^{1/3}U$, $\eta=(\eps^2b^2/a)^{1/3}V$ and time $s=(a\eps b)^{1/3}t$ yield the limiting
Riccati equation $dU/dV=V+U^2$. The normalized fold exit theorem
\cite[Theorem~2.1 and Remarks~2.4 and~2.11]{KS} gives $V=\zeta+o(1)$ in this
orientation. For $|\lambda|\le|p|/2$, $b$ stays uniformly positive, the normalized
coefficients have uniform smooth bounds, and the compact incoming arc gives
uniform contraction. Undoing the scaling gives the formula.
\end{proof}

\subsection{Criteria for the monotone ratio property}\label{sec:criteria}

For an interior member with regular moving ends $z_\pm(h)$, we shoot from a vertical
section inside the moving jump with outward coordinate $\theta=\varsigma h$,
$\varsigma=\pm1$. Lemma~\ref{lem:shooting-annulus} gives the same passages on the
whole closed segment. After regular fast cuts, Proposition~\ref{prop:capture}(b)
provides $T_\pm$ and $S_\pm$. The two leading logarithmic derivatives are
$-\varsigma Q(z_\pm(h))$ in $\theta$. Thus Lemma~\ref{lem:composition} applies if
\begin{equation}\label{eq:shooting-ratio}
 \varsigma[Q(z_-(h))-Q(z_+(h))]<0
\end{equation}
at the limiting shooting height and every continuation satisfies $c_i<\chi_i$.
The same argument applies in rescaled coordinates with the corresponding
$Q$ and small parameter. In the minimum cases the continuations end by
Proposition~\ref{prop:capture}(c), using Lemma~\ref{lem:wedge} on a bounding orbit.

\begin{lemma}\label{lem:corner}
Let $0$ be a nondegenerate minimum, and let $-s<0$ be a quadratic fold at height $A=F(-s)$.
Suppose the left half of $\Gamma_*$ follows $l$ to $(-s,A)$, the upper jump lands regularly
at $z>0$ with $Q_*=Q(z)>0$ and regular fast-jump clearance up to its entry tube, and a bounding
periodic orbit satisfies the two-tube crossing hypothesis of Lemma~\ref{lem:wedge}. Assume that
after the first regular capture the forward continuation to the lower half-line has polynomial
bounds on its first derivative and logarithmic derivative throughout the shooting segment. Then
shooting at $x=\xi_0$ slightly to the right of $-s$ in the coordinate $\theta=h$ gives
$\eps\,\partial_h\log(\Pi_+'/\Pi_-')<-Q_*/2$, and $\Gamma_*$ has the monotone ratio property.
\end{lemma}

\begin{proof}
Let $G(W)=F(\lambda-\sqrt{2W})$ and $W_c=(c-\lambda)^2/2$. Since
$G''(W_{-s})=F''(-s)/(s+\lambda)^2<0$, we may choose $\xi_0>-s$ so close to $-s$ that
$0<G'(W_{\xi_0})<Q_*/8$ uniformly. Fix $y_0=A-\eta_1$ with $0<2\eta_1<A-F(\xi_0)$, and take $\eta>0$
so small that $G''<0$ on $[W_{\xi_0},W_{-s-\eta}]$ and $F(-s-\eta)>A-\eta_1/4$.

\emph{The interval of trajectories.} The upper fast segment crosses $x=\xi_0$ at height $A$
with $A-F(\xi_0)>0$, and the least abscissa of $\Gamma_*$ is $-s$.
By Lemma~\ref{lem:shooting-annulus}, the rightward crossing heights on the whole shooting
segment tend uniformly to $A$, and all its trajectories stay to the right of $-s-\eta/2$
for large $n$. Lemma~\ref{lem:wedge} on a bounding orbit gives $|\lambda|\le M\eps$.

\emph{Backward transition.} In backward time, with $\sigma=-y$ and $\eps W_\sigma=y-G(W)$, the
vector field points into $[W_{\xi_0},W_{-s-\eta}]$ at $W_{\xi_0}$ down to the height $y_0$. Near the
height $A$ the strip contains both $w(y)$ and $l(y)$ (where $w$ is the inverse branch of $F$ immediately to the left of the fold $-s$; see Section~\ref{sec:min-setup}), but for $y_0\le y\le A-\eta_1/2$ the only root
in the strip is $l(y)$, since $F(-s-\eta)>A-\eta_1/4$. On this height interval, outside a fixed
tube about $l$, on which $G'\ge g_0>0$, the fast field points toward the tube with speed at least
$c/\eps$, so all the trajectories under consideration lie in the tube for
$y_0\le y\le A-3\eta_1/4$. On the rest of the strip $G'\ge-K\eta$. Taking $h-y_0\le2\eta_1$ and
then decreasing $\eta$, we obtain
\[
\int_{y_0}^hG'(W)\,dy\ge\frac{g_0\eta_1}4-2K\eta\eta_1\ge\frac{g_0\eta_1}8=:k>0.
\]
Since $G''<0$ on the strip and $W_h>0$, the entry identity~\eqref{eq:entry-var} gives
\[
0<U'(h)\le K\eps^{-1}e^{-k/\eps},\qquad
\eps\,\frac{U''(h)}{U'(h)}=\frac{\eps}{h-F(\xi_0)}-G'(W_{\xi_0})-\int_{y_0}^hG''(W)W_h\,dy>-\frac{Q_*}{8},
\]
and the continuation down to $x=\lambda$ has polynomial distortion.

The forward map contracts at the regular entry to $z(h)$, with $\eps D_{S_+}=-Q(z(h))+o(1)$ by
Proposition~\ref{prop:capture}(b). With $T_-=\mathrm{id}$ and the polynomial continuations, the
one-sided bounds of Lemma~\ref{lem:composition} give the desired estimate, since
$-Q(z(h))+Q_*/8<-Q_*/2$ for $h$ sufficiently close to $A$.
\end{proof}
\section{Classification and the case of a local minimum}\label{sec:minimum}

\subsection{Normalization and classification}\label{sec:min-setup}

The reflection $(x,y,\lambda,F(x))\mapsto(-x,-y,-\lambda,-F(-x))$ preserves time orientation,
$a_5>0$, and the form of~\eqref{eq:system}. It turns a maximum into a minimum. Positive rescaling
of $x,\lambda$, time, and $\eps$ normalizes the first positive fold to $1$ when present. These
changes preserve cycle counts, stability and the sign of $I$. Thus we assume
\[
F(x)=a_2x^2+a_3x^3+a_4x^4+a_5x^5\quad(a_2,a_5>0)\qquad\text{or}\qquad F(x)=a_4x^4+a_5x^5\quad(a_4,a_5>0).
\]
At a nondegenerate minimum, $Q(x)=2a_2+3a_3x+4a_4x^2+5a_5x^3$ has $Q(0)>0$
and positive leading coefficient. With three simple real roots they are all negative
or one is negative and the other two positive. The other cases have a repeated root
or a conjugate pair. For distinct roots $-s<0<1<t$, we speak of the \emph{four-fold case}:
\[
F'(x)=5a_5\,x(x+s)(x-1)(x-t),\qquad s>0,\quad t>1,
\]
with the extremal heights $A=F(-s)$, $B=F(1)$ and $C=F(t)$,
and denote by $w,l,r,v,u$ the inverse branches of $F$ on $(-\infty,-s)$, $(-s,0)$, $(0,1)$,
$(1,t)$ and $(t,\infty)$. When $F$ has only the two extrema $-s<0$, possibly together with an
inflection point with horizontal tangent, we keep the names $w,l$ for the branches on
$(-\infty,-s)$ and $(-s,0)$ and write $r$ for the inverse branch on $(0,\infty)$.

Table~\ref{tab:cases} and Figure~\ref{fig:classification} summarize the
configurations. In the table, $I$ denotes the slow divergence integral and
$\Lambda_i$ denotes a family of candidate slow--fast cycles. Put $n_3=1$ for $a_3\ge0$ and $n_3=2$ otherwise.

\begin{table}[!htb]
\centering\footnotesize
\caption{Configurations of $y=F(x)$ near the limiting equilibrium, after the normalization of
Section~\ref{sec:min-setup}. The listed families contain every admissible slow--fast cycle in the
corresponding case. ``$+/-$'' means that, along this candidate chain, $I$ has exactly one zero, with
positive values before it and negative values after it in the outward order. In row~(j), $e$ is the
shape parameter defined in Section~\ref{sec:res}. The rows correspond to the
panels of Figure~\ref{fig:classification}.}\label{tab:cases}
\renewcommand{\arraystretch}{1.12}
\setlength{\tabcolsep}{4pt}
\begin{tabular}{@{}L{.25\textwidth}L{.24\textwidth}L{.15\textwidth}L{.11\textwidth}cc@{}}
\toprule
Critical curve & Candidate cycles, inside to outside & Sign of $I$ & Shrinking cycles & $N_F\le$ & Section\\
\midrule
\multicolumn{6}{@{}l}{\emph{Nondegenerate minimum at $x=0$, all other local extrema at negative abscissae}}\\
(a) four quadratic folds & cycles containing the turning point (canard or common) & $-$ & attracting & $1$ & \ref{sec:direct}\\
(b) two quadratic extrema, possibly with a stationary inflection & chain $\Lambda_1,\Lambda_2$ & $-$ if $a_3\ge0$; $+/-$ if $a_3<0$ & as in Prop.~\ref{prop:shrinking} & $n_3$ & \ref{sec:two}\\
(c) remote quartic extremum & all admissible cycles & $-$ & attracting & $1$ & \ref{sec:direct}\\
\midrule
\multicolumn{6}{@{}l}{\emph{Nondegenerate minimum at $x=0$, folds $-s<0<1<t$}}\\
(d) $A<B$ & chain $\Lambda_1,\Lambda_2$ & as in (b) & as in (b) & $n_3$ & \ref{sec:two}\\
(e) $A=B$ & chain $\Lambda_1,\Lambda_2$ & $+/-$ & repelling & $2$ & \ref{sec:equal-max}\\
(f) $A>B$, $C>0$ & chain $\Lambda_1,\dots,\Lambda_5$ & $+/-$ & repelling & $2$ & \ref{sec:five}\\
(g) $A>B$, $C\le0$ & inner and outer cycles & $+$ inner, $-$ outer & repelling & $2$ & \ref{sec:term}\\
\midrule
(h) quartic minimum at $x=0$ & cycles containing the turning point (canard or common) & $-$ & attracting & $1$ & \ref{sec:direct}\\
\midrule
\multicolumn{6}{@{}l}{\emph{$F$ strictly monotone near $x=0$}}\\
(i) two extrema & common cycles & strict & none & $1$ & \ref{sec:common}\\
(j) four quadratic folds & common cycles; at a resonance, chain $\Lambda_1,\Lambda_2,\Lambda_3$ if $e\le1/\sqrt5$ and no admissible cycles if $e>1/\sqrt5$ & strict; $+/-$ & none & $2$ & \ref{sec:common}, \ref{sec:res}\\
(k) no extrema & none & --- & none & $0$ & \ref{sec:common}\\
\bottomrule
\end{tabular}
\end{table}

The members of a candidate chain are ordered by $\preceq$, with consecutive
families sharing their boundary members.

\begin{theorem}\label{thm:minimum}
Let $x=0$ be a nondegenerate or a quartic local minimum of $F$. Then one of the following holds.
\begin{enumerate}[label=(\roman*)]
\item Every admissible slow--fast cycle has $I<0$, and $N_F(\eps,\lambda)\le1$ for all
sufficiently small $\eps>0$ and $|\lambda|$.
\item $x=0$ is a nondegenerate minimum with $a_3<0$, and the set $\mathcal A$ of admissible
slow--fast cycles is the disjoint union of sets $\mathcal A_{\mathrm{in}}$, $\mathcal A_0$ and
$\mathcal A_{\mathrm{out}}$ such that: $I>0$ on $\mathcal A_{\mathrm{in}}$ and $I<0$ on $\mathcal A_{\mathrm{out}}$;
$\mathcal A_0$ is empty or consists of one neutral cycle with the monotone ratio property; and
$\Gamma\preceq\Gamma'$ whenever $\Gamma\in\mathcal A_{\mathrm{in}}$ and
$\Gamma'\in\mathcal A_0\cup\mathcal A_{\mathrm{out}}$, or $\Gamma\in\mathcal A_0$ and
$\Gamma'\in\mathcal A_{\mathrm{out}}$.
\end{enumerate}
Alternative (i) occurs in rows (a), (c) and (h) of Table~\ref{tab:cases} and in rows (b) and (d)
when $a_3\ge0$; alternative (ii) occurs in the remaining rows (b), (d), (e), (f) and (g).
\end{theorem}

In alternative (ii), periodic orbits shrinking to the equilibrium are eventually hyperbolic and
repelling by Proposition~\ref{prop:shrinking}. The proof of Theorem~\ref{thm:minimum} occupies the
rest of this section and is completed in Section~\ref{sec:proof-minimum}.

\subsection{Comparison of equal-height roots}\label{sec:equal-height}

The canard without head at height $h$ has integral
\begin{equation}\label{eq:J}
J(h)=\cH(l(h))-\cH(r(h)),\qquad J'(h)=Q(l(h))-Q(r(h)).
\end{equation}
This is~\cite[Theorem~2.1, (2.6)]{LiLu} in our time orientation. The height derivative extends
continuously across a stationary inflection by integration in $h$.

\begin{lemma}\label{lem:J}
Suppose $0$ is a nondegenerate minimum and that $F$ has either two extrema or four quadratic
folds ordered as $-s<0<1<t$. Set $h_1=A$ in the first case and $h_1=\min(A,B)$ in the second.
\begin{enumerate}[label=(\alph*)]
\item $J'$ vanishes at most once on $(0,h_1)$.
\item As $h\downarrow0$,
\begin{gather*}
J'(h)=-6a_3\sqrt{h/a_2}+O(h)\quad\text{if }a_3\ne0,\\
J'(h)=-10a_5(h/a_2)^{3/2}+O(h^2)\quad\text{if }a_3=0 .
\end{gather*}
\item In the four-fold case, $A\ge B$ implies $a_3<0$.
\item In the four-fold case with $A>B$, $J'\ge0$ on $(0,B)$ and $J>0$ on $(0,B]$.
\end{enumerate}
\end{lemma}

\begin{proof}
Part (a) is Proposition~\ref{prop:LL}(a). For (b), $l(h),r(h)=\mp\sqrt{h/a_2}+O(h)$, and
$J'=3a_3(l-r)+4a_4(l^2-r^2)+5a_5(l^3-r^3)$. If $a_3=0$, the odd part of $F$ is $a_5x^5$, so
$F(l)=F(r)$ gives $l^2-r^2=O(h^{5/2})$ and $l+r=O(h^2)$, and the second expansion follows. For (c),
the sign of $A-B$ is that of $3s^2+(5t-4)s+3-5t$, which increases in $s>0$ and is negative at
$s=t/(1+t)$, while $3a_3=5a_5[t-s(1+t)]$. Hence $A\ge B$ forces $s>t/(1+t)$, that is, $a_3<0$. For
(d), $a_3<0$ gives $J'>0$ near $0$ by (b), and $J'(B-)=Q(l(B))>0$ because $r(h)\to1$ and $Q(1)=0$.
By (a), $J'$ has no sign change on $(0,B)$, so $J'\ge0$ there and $J>0$ on $(0,B]$.
\end{proof}

\subsection{Direct cases}\label{sec:direct}

Figure~\ref{fig:extremal-direct} shows the three configurations with a direct root comparison.

\begin{figure}[tbp]
\centering
\PlotPanel{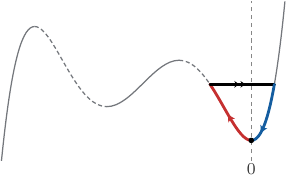}{Outer quadratic fold}{fig:extremal-direct-a}\hfill
\PlotPanel{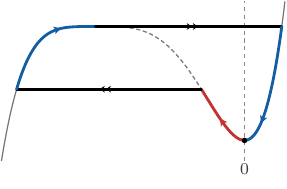}{Remote quartic fold}{fig:extremal-direct-b}\hfill
\PlotPanel{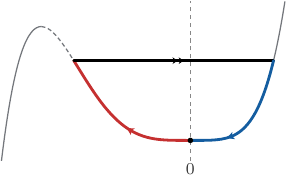}{Quartic fold at $x=0$}{fig:extremal-direct-c}
\caption{Candidate slow--fast cycles in rows (a), (c) and (h) of Table~\ref{tab:cases}. Here and in
the following figures the whole critical curve is drawn; solid and dashed gray branches are
attracting and repelling. On candidate slow--fast cycles, blue and red denote attracting and repelling slow arcs, respectively.
Fast segments are black. Single and double arrowheads indicate slow and fast motion, respectively. The dashed vertical line is $x=0$. The curves are
schematic and are not computed periodic orbits.}
\label{fig:extremal-direct}
\end{figure}

\begin{proposition}\label{prop:direct-min}
Suppose that $F$ is of one of the following forms.
\begin{enumerate}[label=(\alph*)]
\item $F'(x)=5a_5 x(x+p_1)(x+p_2)(x+p_3)$ with $0<p_1<p_2<p_3$.
\item $F'(x)=k_1 x(x+p)^3$ with $k_1,p>0$.
\item $F(x)=a_4x^4+a_5x^5$ with $a_4,a_5>0$.
\end{enumerate}
In each case every admissible slow--fast cycle has $I<0$ and $N_F\le1$ for all
sufficiently small parameters.
\end{proposition}

\begin{proof}
The shrinking cycles are eventually hyperbolic and attracting by
Proposition~\ref{prop:shrinking}. In (a) all the coefficients $a_2,\dots,a_5$ are
positive. In (b), $3a_3=3k_1p^2>0$, and $Q=k_1(x+p)^3$ is strictly increasing,
giving $I<0$ by~\eqref{eq:sdi-height}. For $\lambda\ge0$ in (b), all three
coefficients in~\eqref{eq:odd-part} have nonnegative sign and the quintic term
is positive, excluding periodic orbits by Lemma~\ref{lem:odd}.

For (a), $F$ increases on $[0,\infty)$, so the right-hand estimate of
Proposition~\ref{prop:compact} gives $y_-^*=0$. At $y>0$, $Q(x^+)>0$ on the positive inverse. A
negative root has $Q(x^-)<Q(x^+)$ wherever $Q(x^-)\le0$ or $x^-\in(-p_1,0)$, since $Q$ increases
on $(-p_1,\infty)$. For a root in $(-p_3,-p_2)$, the comparison follows from
Proposition~\ref{prop:LL}(c) if $F(-p_2)\ge0$. If $F(-p_2)<0$, reaching that branch requires a
leftward crossing of $x=-p_2$ below a negative height, contrary to $y_-^*=0$.

For (c), the coefficients in~\eqref{eq:odd-part} exclude $\lambda\ge0$. For small $\lambda<0$,
one has $F\ge0$ on $x\ge\lambda$ and $F\le F(-4a_4/(5a_5))$ on $x\le\lambda$, so $y_-^*=0$ and
$y_+^*\le F(-4a_4/(5a_5))$. For an equal-height pair $-u<0<v$ at positive height,
$a_4(u^4-v^4)=a_5(u^5+v^5)$ gives $u>v$ and
\[
 4a_4(u^2-v^2)=4a_5\,\frac{u^5+v^5}{u^2+v^2}<5a_5(u^3+v^3).
\]
Thus $Q(-u)<Q(v)$ also in (c). Now~\eqref{eq:sdi-height} and
Corollary~\ref{cor:one-type} give the conclusions.
\end{proof}

\subsection{A chain of two families}\label{sec:two}

Figure~\ref{fig:two-piece} shows the chain of this section and its two exceptional landings.

\begin{figure}[tbp]
\centering
\PlotPanel{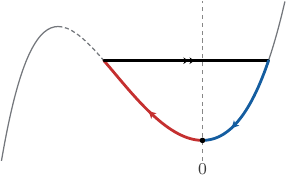}{Two extrema: $\Lambda_1$}{fig:two-piece-a}\hspace{.045\textwidth}%
\PlotPanel{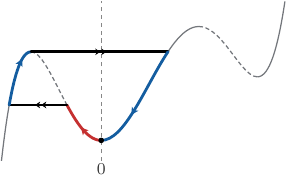}{$A<B$: $\Lambda_2$}{fig:two-piece-b}
\par\medskip
\PlotPanel{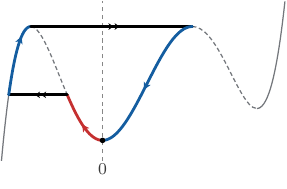}{$A=B$: $\Lambda_2$}{fig:two-piece-c}\hspace{.045\textwidth}%
\PlotPanel{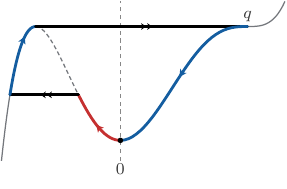}{Landing at an inflection point: $\Lambda_2$}{fig:two-piece-d}
\caption{The chain $\Lambda_1,\Lambda_2$. Panel~(a) shows a canard cycle without head for a
quintic with two extrema; panels~(b)--(d) show canard cycles with a left head when $A<B$, when
$A=B$, and when the upper fast segment lands at an inflection point $q$ with horizontal tangent.
The landings in~(c) and~(d) are treated in Sections~\ref{sec:equal-max} and~\ref{sec:inflection}.}
\label{fig:two-piece}
\end{figure}

\begin{proposition}\label{prop:two}
Let $x=0$ be a nondegenerate minimum, and suppose either that $Q$ has one simple negative root $-s$
and either a pair of nonreal roots or a nonzero double root, or that we are in the four-fold case
with $A\le B$. Then every admissible slow--fast cycle belongs to the chain
\begin{equation}\label{eq:two-chain}
\begin{aligned}
I_1(h)&=J(h)=\cH(l(h))-\cH(r(h)), && 0<h\le A,\\
I_2(h)&=\cH(l(h))-\cH(r(A))+\cH(-s)-\cH(w(h)), && A\ge h\ge0,
\end{aligned}
\end{equation}
where $\Lambda_1(h)$ is the slow--fast cycle of height $h$ through the turning point and
$\Lambda_2(h)$ is the slow--fast cycle whose lower left fast segment is at height $h$. For $0<h<A$,
$\Lambda_1(h)$ is a canard cycle without head and $\Lambda_2(h)$ is a canard cycle with a left head;
the two families share the junction $\Lambda_1(A)=\Lambda_2(A)$, where the additional left fast
segment has zero length, and $\Lambda_2(0)$ is a common cycle with attracting slow arcs on $w$ and $r$. The chain
is ordered from
inside to outside by increasing $h$ on $\Lambda_1$ and decreasing $h$ on $\Lambda_2$. If $a_3\ge0$,
then $I<0$ on the chain and $N_F(\eps,\lambda)\le1$. If $a_3<0$, the chain contains exactly one
neutral cycle, $I$ is positive on the members inside it and negative on those outside it, and the
neutral cycle has the monotone ratio property.
\end{proposition}

\begin{proof}
\emph{Admissible cycles.} In the two-extrema case, $F$ is increasing on $[0,\infty)$, possibly with
an inflection point with horizontal tangent, so $r$ is the only positive branch, and the negative
branches are $w<-s<l<0$. The right-half estimate of Proposition~\ref{prop:compact} gives $y_-^*=0$, and since $F\le A$ on
$x\le0$, $0<y^*_+\le A$. The jump condition shows that the left half follows $l$ if $y_+^*<A$, and
may switch once from $l(h)$ to $w(h)$ if $y_+^*=A$; the right half follows $r$. In the four-fold
case with $A<B$, again $y_+^*\le A<B$; choosing $r_0\in(0,1)$ with $A<F(r_0)<B$, crossing $x=r_0$
would require $y_+\ge F(r_0)$, so the branches $v$ and $u$ are not reached. For $A=B$,
Section~\ref{sec:equal-max} shows that $\max x\le1+o(1)$, so again only $r$ is reached on the
right. This gives the chain~\eqref{eq:two-chain}; for $0<h<A$ the derivatives are
\[
I_1'(h)=Q(l(h))-Q(r(h)),\qquad I_2'(h)=Q(l(h))-Q(w(h))>0,
\]
and $I_2(0)=\int_0^A[Q(w)-Q(r)]\,dy<0$. The ordering along the chain follows from
Definition~\ref{def:order}, since $w<l$ and $l$, $r$ are monotone.

\emph{Signs.} At a regular landing $I_1'(A-)=-Q(r(A))<0$. If $a_3\ge0$, then
$I_1'<0$ near zero and the landing is regular, as shown below. Lemma~\ref{lem:J} then gives
$I_1<0$ and $I_2\le I_1(A)<0$. Corollary~\ref{cor:one-type} applies since the shrinking cycles
attract. If $a_3<0$, then $I_1'>0$ near zero and has at most one zero. Any interior zero of $I_1$
is a strict downward crossing. As $I_2$ decreases outward from $I_1(A)$ to $I_2(0)<0$, the chain
has one neutral member, with strict signs on either side. The special landings $A=B$ and $A=F(q)$
considered below both give $I_1(A)>0$.

For an inflection point with horizontal tangent write $Q(x)=k(x+s)(x-q)^2$.
If $q<0$, then $3a_3=kq(q-2s)>0$. For $q>0$ with $F(q)\le A$, the identity
\[
 A-F(q)=k(s+q)^4(3s-2q)/60
\]
gives $s\ge2q/3$, so $3a_3=kq(q-2s)<0$ and $I_1'>0$ near zero. If $F(q)<A$, then
$I_1'(F(q))=Q(l(F(q)))>0$, and the single-zero bound gives $I_1(F(q))>0$. Equality $F(q)=A$ is
handled in Section~\ref{sec:inflection}. Thus a moving endpoint of the neutral member is regular.

\emph{Shooting.} On $\Lambda_1$ the coordinate is $h$ and the first captures
are on $r$ forward and $l$ backward. The difference in~\eqref{eq:shooting-ratio} is $I_1'$. On
$\Lambda_2$ the coordinate is $-h$, the captures are on $w$ forward and $l$ backward, and the
difference is $-I_2'$. Both are strictly negative at interior neutral members. On $\Lambda_2$ the
forward continuation first follows a compact ascending arc of $w$, then passes the fold $-s$ by
Lemma~\ref{lem:corridor} to a cut just beyond it. The remaining upper jump has regular clearance
up to the right receiving tube. A regular landing on $r$ gives polynomial bounds by
Proposition~\ref{prop:capture}(b). If a stationary inflection follows that entry,
Lemma~\ref{lem:corridor} applies from a preceding regular arc through the inflection and back
into the regular tube. The other two receiving passages are treated below. The lower
continuations on $l,r$ use Proposition~\ref{prop:capture}(c). If the junction
$\Lambda_1(A)=\Lambda_2(A)$ is neutral, its right landing is regular. Indeed,
Sections~\ref{sec:equal-max} and~\ref{sec:inflection} show that $J(A)>0$ in both exceptional cases
($A=B$ and $A=F(q)$). These same right-hand bounds therefore allow Lemma~\ref{lem:corner} at a neutral
junction.
\end{proof}

\subsubsection{Equal maxima}\label{sec:equal-max}

Let $A=B$ and define $\Delta_L=\zeta(s^2/\varkappa_{-s})^{1/3}$ and
$\Delta_R=\zeta(1/\varkappa_1)^{1/3}$. The factorization $F(x)-A=a_5(x+s)^2(x-1)^2(x-z)$ has
$z>1$. The condition $F'(0)=0$ gives $2/s-2-1/z=0$, so $s=2z/(2z+1)<1$. Hence
\[
  (\Delta_L/\Delta_R)^3=s^2(z-1)/(z+s)<1 .
\]
For a periodic orbit with
$y_+>A-\eta$, comparison with a reference trajectory that starts at height $A-\eta$ at a point
strictly to the left of $w(A-\eta)$ and of all the trajectories under consideration (it enters
along $w$), the order of solutions of the scalar equations~\eqref{eq:half}, and
Proposition~\ref{prop:fold-exit} give $y_+\le A+\Delta_L\eps^{2/3}+o(\eps^{2/3})$. On the other
hand, fix $c\in(1,t)$ and follow backward in time the orbit tangent to $x=c$ at $(c,F(c))$, past
the fold at $(1,A)$, to $x=\lambda$; its upper intercept is $Y_c=A+\Delta_R\eps^{2/3}+o(\eps^{2/3})$. This
reference orbit is a graph $x_c(y)\le c$ on $[F(c),Y_c]$, since $\ddot x=-\eps(x-\lambda)<0$ at its
rightmost points, so every orbit through $(\lambda,Y)$ reaching $x=c$ has $Y\ge Y_c$. Since
$\Delta_L<\Delta_R$, no such orbit reaches the fixed section $x=c$ for small $\eps$.
An orbit with $y_+\le A-\eta$ cannot cross $x=1$ to the right. Letting $c\downarrow1$
gives $\max x\le1+o(1)$ for periodic orbits. Moreover, as
$h\uparrow A$,
\[
J'(h)=2\bigl(\sqrt{\varkappa_{-s}}/s-\sqrt{\varkappa_1}\bigr)\sqrt{A-h}+O(A-h)>0,
\]
the coefficient being positive because $\Delta_L<\Delta_R$. Since $a_3<0$ by Lemma~\ref{lem:J}(c),
also $J'>0$ near $0$, and Lemma~\ref{lem:J}(a) gives $J'\ge0$ on $(0,A)$. Hence $J(A)>0$ and the
neutral cycle lies in the interior of $\Lambda_2$.

Only the receiving passage on the right of the upper jump remains to be checked. Choose a cut
$x=1-\delta$ on that jump and a height $y_r=A-d$ on $r$, with $d>0$ fixed and small. Choose
$0<x_0<1-\delta$ with $F(x_0)<y_r$. The field points inward at $x=x_0$, and $\max x\le1+o(1)$
confines the whole annulus to $x\le1+\gamma$ for each fixed $\gamma>0$ and large $n$. There
$G'\ge-K\gamma$ and $|G''|\le K$. The entry heights tend uniformly to $A$,
so~\eqref{eq:receiving-strip} gives the exponent $c_+=2dK\gamma$. Choose $\gamma$ so that
$c_+<\chi_+$ from the first capture on $w$. At $y_r$ the right half is in the regular tube of
$r$. The other factors have polynomial bounds, so~\eqref{eq:shooting-ratio} applies.

\subsubsection{Landing at an inflection point with horizontal tangent}\label{sec:inflection}

A quadratic fold $p$ and an inflection point $q$ with horizontal tangent at the same height $A$
satisfy $F(x)-A=a_5(x-p)^2(x-q)^3$ and $F'(x)=5a_5(x-q)^2(x-p)(x-b_0)$ with $b_0=(3p+2q)/5$. When
the minimum $b_0$ is at the origin, the normalization $p=-2$, $q=3$, $5a_5=1$ gives
\[
F'(x)=x(x+2)(x-3)^2,\qquad F(x)-A=\tfrac15(x+2)^2(x-3)^3,\qquad F(0)=0 .
\]
By Proposition~\ref{prop:LL}(d), $J'$ has no zero on $(0,A)$. Since $a_3<0$, $J'>0$ on $(0,A)$, so
$J(A)>0$ and the neutral cycle lies in $\Lambda_2$.

For the upper receiving passage take $0<\xi<q$ as the right cut. The factorization gives positive
clearance $A-F(x)>0$ from a cut beyond $p$ to $\xi$. Choose $y_1=A-d_1>0$ and walls
$0<x_0<\xi<x_1$ with $F(x_0)<y_1$ and $F(x_1)>A+d_1$. Both walls point inward from the entry
height $A+o(1)$ to $y_1$, with $G'\ge0$ and $|G''|\le K$ on the strip.
Thus~\eqref{eq:receiving-strip} with $\mu=0$ gives polynomial bounds, and the continuation starts
in the regular tube of $r$ below $q$. In particular both continuation exponents vanish
in~\eqref{eq:shooting-ratio}.

\subsection{A chain of five families: \texorpdfstring{$A>B$ and $C>0$}{A>B and C>0}}\label{sec:five}

Let $A>B>C>0$. The left half switches from $l(h)$ to $w(h)$ only if $y_+^*=A$, and $y_-^*=0$. On
the right only $r$ exists below $C$ and only $u$ above $B$. In increasing height the allowed
switches are $r\to v$, $u\to v$ for $C<y<B$, and $r(C)\to t$ at $C$. Below top height $B$ the
upper segment cannot reach $v$ or $u$. An arc on $v$ ends at the fold of height $B$, whereas with
top height above $B$ the right half switches to $u$ at $C$. These possibilities give the ordered
chain in
Figure~\ref{fig:five-piece}.

\begin{figure}[tbp]
\centering
\PlotPanel{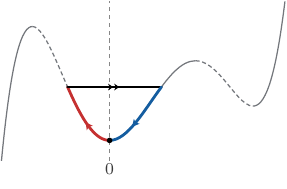}{$\Lambda_1$}{fig:five-piece-a}\hfill
\PlotPanel{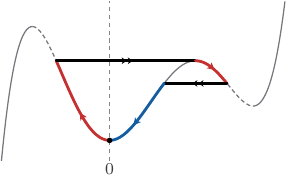}{$\Lambda_2$}{fig:five-piece-b}\hfill
\PlotPanel{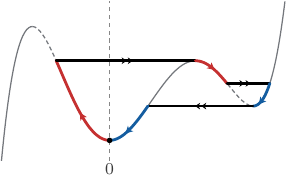}{$\Lambda_3$}{fig:five-piece-c}
\par\medskip
\PlotPanel{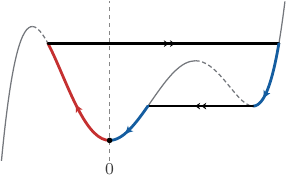}{$\Lambda_4$}{fig:five-piece-d}\hspace{.035\textwidth}%
\PlotPanel{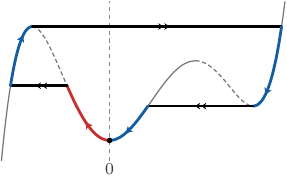}{$\Lambda_5$}{fig:five-piece-e}
\caption{The chain $\Lambda_1,\dots,\Lambda_5$ for $A>B>C>0$, drawn for one polynomial $F$ on a
common scale. The horizontal marks give the fold heights $A=F(-s)$, $B=F(1)$ and $C=F(t)$.
Each panel shows an interior member of its family. Consecutive families share their boundary members.}
\label{fig:five-piece}
\end{figure}

\begin{center}\small
\renewcommand{\arraystretch}{1.15}
\begin{tabular}{@{}llll@{}}
\toprule
Family & Top height & Right half, in increasing height & Outward parameter\\
\midrule
$\Lambda_1(h)$ & $h\in(0,B]$ & $r$ & $h$ increasing\\
$\Lambda_2(\kappa)$ & $B$ & $r$, then $r\to v$ at $\kappa\in[C,B]$ & $\kappa$ decreasing\\
$\Lambda_3(\kappa)$ & $B$ & $r\to t$ at $C$, then $u\to v$ at $\kappa\in[C,B]$ & $\kappa$ increasing\\
$\Lambda_4(h)$ & $h\in[B,A]$ & $r\to t$ at $C$, then $u$ & $h$ increasing\\
$\Lambda_5(h)$ & $A$ & as for $\Lambda_4(A)$; on the left $l\to w$ at $h\in[0,A]$ & $h$ decreasing\\
\bottomrule
\end{tabular}
\end{center}

\begin{proposition}\label{prop:five}
In the four-fold case with $A>B>C>0$, the chain $\Lambda_1,\dots,\Lambda_5$ contains exactly one
neutral cycle, $I$ is positive on the members inside it and negative on those outside it, and the
neutral cycle has the monotone ratio property.
\end{proposition}

\begin{proof}
By Lemma~\ref{lem:J}, $a_3<0$ and $I_1=J>0$ on $(0,B]$. Differentiation at an interior jump
height $C<\kappa<B$, as in~\cite[Theorem~2.1]{LiLu}, gives
\[
I_2'(\kappa)=Q(v(\kappa))-Q(r(\kappa))<0,\qquad I_3'(\kappa)=Q(v(\kappa))-Q(u(\kappa))<0 .
\]
Hence $I_2$ increases along the chain from $J(B)>0$ to $I_2(C)$, and $I_3$ decreases from
$I_3(C)=I_2(C)$ to $I_3(B)=I_4(B)$, which is positive by Proposition~\ref{prop:cert} (a
computer-assisted inequality; see \ref{app:cert}). Further,
\[
I_4(h)=\cH(l(h))-\cH(u(h))+\cH(t)-\cH(r(C)),\qquad I_5(h)=I_4(A)+\cH(l(h))-\cH(w(h)).
\]
By Proposition~\ref{prop:LL}(b), $I_4$ has at most one zero on $(B,A)$ and it is simple. Since
$I_4(B)>0$ and $I_4'(A-)=-Q(u(A))<0$, while $I_5$ decreases outward to $I_5(0)<0$, there is
exactly one neutral member. On $\Lambda_4$, we shoot rightward with coordinate $\theta=h$, taking the first captures on $u$
in forward time and on $l$ in backward time; on $\Lambda_5$, we shoot leftward with $\theta=-h$ and
first captures on $w$ and $l$. The respective differences in~\eqref{eq:shooting-ratio} are
$Q(l)-Q(u)=I_4'$ and $Q(w)-Q(l)=-I_5'$, strictly negative at interior zeros.

The forward continuation from $w$ passes $-s$ by Lemma~\ref{lem:corridor} from a fixed compact
arc and jumps regularly to $u(A)$ since $A>B$. After either capture on $u$, the continuation
proceeds in three steps: first, along a compact descending arc of $u$ (starting in $(C,B)$, or
above $B$ if needed), it acquires a positive speed lower bound; second, Lemma~\ref{lem:corridor}
carries it past the fold at $t$ down to a fixed height $y_2\in(0,C)$, where $V>0$ implies
$G(W)>y\ge y_2$ and keeps the orbit away from $x=\lambda$ with bounded $G'$ and $G''$; third,
since $r$ is the unique right root below $C$, the positive fast gap outside a tube about $r$ on a
fixed height interval below $C$ forces entry into the $r$-tube within $O(\eps)$ in height, after
which the continuation follows $r$ to the lower half-line. All continuation bounds are polynomial.
At the junction the same argument on the right after capture on $u(A)$ verifies
Lemma~\ref{lem:corner}.
\end{proof}

\subsection{\texorpdfstring{$A>B$ and $C\le0$}{A>B and C<=0}: inner repelling and outer attracting cycles}\label{sec:term}

Figure~\ref{fig:terminal} shows the candidate limiting cycles in this case.

\begin{figure}[tbp]
\centering
\PlotPanel{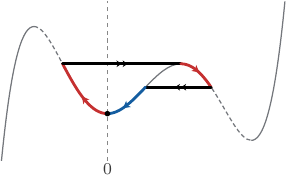}{$C<0$, $\bar x<t$}{fig:terminal-a}\hspace{.045\textwidth}%
\PlotPanel{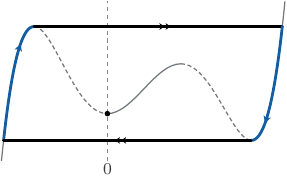}{$C<0$, $\bar x>t$}{fig:terminal-b}
\par\medskip
\PlotPanel{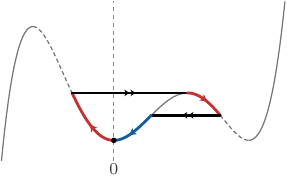}{$C=0$, $\bar x<t$}{fig:terminal-c}\hspace{.045\textwidth}%
\PlotPanel{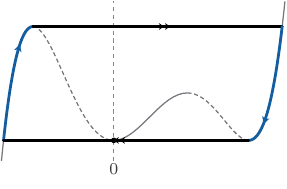}{$C=0$, $\bar x>t$}{fig:terminal-d}
\caption{Candidate limiting cycles for $A>B$ and $C\le0$. Panels~(a,c) show inner cycles with
rightmost abscissa $\bar x<t$; periodic orbits converging to them are eventually repelling.
Panels~(b,d) show the outer common cycle with $\bar x>t$; periodic orbits converging to it are
eventually attracting. For $C=0$ in~(d), the lower fast segment lands at $w(0)$ at height zero. The
boundary value $\bar x=t$ belongs to the inner, repelling, regime.}
\label{fig:terminal}
\end{figure}

\begin{proposition}\label{prop:term}
In the four-fold case with $A>B$ and $C\le0$, let $\bar x$ be the rightmost abscissa of an
admissible slow--fast cycle. If $\bar x\le t$, then its slow divergence integral is positive; if
$\bar x>t$, it is negative. Every cycle of the first kind precedes every cycle of the second kind
for $\preceq$.
\end{proposition}

\begin{proof}
If $\bar x\le t$, the right profile gives $y_+^*\le\max_{[0,t]}F=B$. When $C=0$, the
factorization $F(x)=a_5x^2(x-t)^2(x-w(0))$ and the right-half estimate of
Proposition~\ref{prop:compact} give $y_-^*=0$. If $C<0$, a negative bottom height forces a
leftward jump to $w$ and then $y_+^*\ge A>B$, so again $y_-^*=0$. The left half follows $l$, and
on $r,v$ we have $Q(v)\le0<Q(r)$. Hence
\[
I\ge\int_0^{y_+^*}\bigl[Q(l)-Q(r)\bigr]\,dy=J(y_+^*)>0,
\]
including the case $\bar x=t$, where $y_+^*=B$: then the approximating periodic orbits cross $x=1$
from left to right for all large $n$, so $y_+^*\ge F(1)=B$.

Let $\bar x>t$ and $C<0$. The cycle follows $u$ down to $(t,C)$, jumps at the height $C<0$ across
$x=0$ to $w(C)$, follows $w$ up to $(-s,A)$ and jumps at the height $A$ to $u(A)$. It is a common
cycle with attracting slow arcs, so $I<0$.

If $\bar x>t$ and $C=0$, write $F(x)=a_5x^2(x-t)^2(x-z_0)$ with $z_0=w(0)<-s$. The points of
maximal abscissa approach a compact regular subarc of $u$ away from $t$. By
Proposition~\ref{prop:fold-exit} the forward orbit exits the fold at height
$-\zeta(t^2/\varkappa_t)^{1/3}\eps^{2/3}+o(\eps^{2/3})$. The regular jump to a small fixed
$x=\delta>0$ changes the height by $O(\eps)$. Take $0<\gamma<\zeta(t^2/\varkappa_t)^{1/3}$ and
$m_0>0$ with $F(x)\ge m_0x^2$ for $|x|\le\delta$. Up to a first crossing of the height
$-\gamma\eps^{2/3}/2$, the change across this strip is bounded by
\[
|\Delta y|\le\eps\int_{-\delta}^{\delta}
 \frac{|x|+|\lambda|}{\gamma\eps^{2/3}/2+m_0x^2}\,dx
 =O(\eps|\log\eps|+|\lambda|\eps^{2/3})=o(\eps^{2/3}).
\]
The bound excludes such a crossing, so the orbit passes below zero and lands at $w(0)$ in the
limit. The attracting limit on $w,u$ is common and has $I<0$.

Finally, the inner limits have top height at most $B$, with halves on
$l$ and $r\cup v$. The outer limit has intercepts $C\le0$ and $A>B$ and
halves on $w$ and $u$. Since $w<l$ and $v<t<u$, the claimed order follows.
\end{proof}

\subsection{Proof of Theorem~\ref{thm:minimum}}\label{sec:proof-minimum}

Propositions~\ref{prop:direct-min} and~\ref{prop:two} give (i) in rows (a), (c), (h) and in rows
(b), (d) with $a_3\ge0$. In the other rows $a_3<0$ by assumption or Lemma~\ref{lem:J}(c). In rows
(b), (d), (e) and (f), let $\Gamma_0$ be the unique neutral member of the candidate chain in
Proposition~\ref{prop:two} or~\ref{prop:five}. On the admissible set $\mathcal A$, put
\[
\mathcal A_0=\mathcal A\cap\{\Gamma_0\},\qquad
\mathcal A_{\mathrm{in}}=\{\Gamma\in\mathcal A:I(\Gamma)>0\},\qquad
\mathcal A_{\mathrm{out}}=\{\Gamma\in\mathcal A:I(\Gamma)<0\}.
\]
The order and the monotone ratio property follow from those propositions.
In row (g), Proposition~\ref{prop:term} gives the order and the signs with $\mathcal A_0=\emptyset$.
\qed
\section{The locally monotone case}\label{sec:mono}

In this section $F(0)=0$ and $F$ is strictly monotone near $x=0$, possibly with $F'(0)=0$. By
Lemma~\ref{lem:trap}, no periodic orbit enters the rectangle $\mathcal R_0$, so none shrinks to
the equilibrium, every periodic orbit crosses all heights in $(-h_0,h_0)$, and
Proposition~\ref{prop:sdi} applies to every convergent sequence of periodic orbits.

\begin{theorem}\label{thm:mono}
Let $F$ be strictly monotone near $x=0$. Then one of the following holds.
\begin{enumerate}[label=(\roman*)]
\item $F$ has at most two extrema, every admissible slow--fast cycle is common, and
$N_F\le1$ for all sufficiently small parameters. If $F$ has no extremum, $N_F=0$.
\item $F$ has four quadratic folds and every admissible slow--fast cycle is common. Moreover,
$N_F(\eps,\lambda)\le2$ for all sufficiently small $\eps>0$ and $|\lambda|$.
\item $F$ has four quadratic folds and satisfies the resonance condition~\eqref{eq:resonance}
below, and the set of admissible slow--fast cycles has a decomposition as in
Theorem~\ref{thm:minimum}(ii).
\end{enumerate}
\end{theorem}

\subsection{Common cycles and root counting}\label{sec:common}

Figure~\ref{fig:nonextremal} shows typical common slow--fast cycles in this case. When all contact
points are quadratic jump points, the corresponding slow--fast cycles were analyzed by Li and
Lu~\cite[Lemma~3.4 and Figure~14(a)--(c)]{LiLu}.

\begin{figure}[tbp]
\centering
\PlotPanel{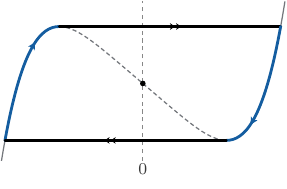}{Two quadratic extrema}{fig:nonextremal-a}\hspace{.045\textwidth}%
\PlotPanel{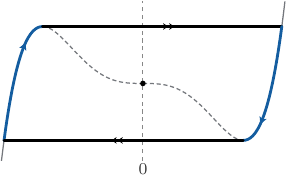}{Inflection point with horizontal tangent at $x=0$}{fig:nonextremal-b}
\par\medskip
\PlotPanel{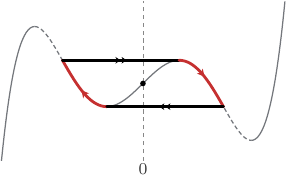}{Repelling common cycle}{fig:nonextremal-c}\hspace{.045\textwidth}%
\PlotPanel{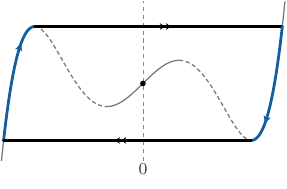}{Attracting common cycle}{fig:nonextremal-d}
\caption{Common slow--fast cycles when $F$ is strictly monotone near $x=0$. Panels~(a,b) show
attracting common cycles when $F'(0)\ne0$ and when $x=0$ is an inflection point with horizontal
tangent; panels~(c,d) show, for one quintic with four extrema, a repelling inner and an attracting
outer common cycle.}
\label{fig:nonextremal}
\end{figure}

\begin{lemma}\label{lem:common-sign}
Every admissible common slow--fast cycle $\Gamma$ satisfies $I(\Gamma)\ne0$: $I(\Gamma)<0$ if its
slow arcs are attracting and $I(\Gamma)>0$ if they are repelling.
\end{lemma}

\begin{proof}
By Lemma~\ref{lem:trap}, both halves of $\Gamma$ cross the heights $(-h_0,h_0)$ on slow arcs away
from $x=0$. On attracting arcs $F'>0$, so $Q(x_-)<0<Q(x_+)$, and the integrand
of~\eqref{eq:sdi-height} is negative wherever both halves are on slow arcs; on repelling arcs it is
positive.
\end{proof}

\begin{proposition}\label{prop:common}
Suppose that every admissible slow--fast cycle is common. Then $N_F(\eps,\lambda)\le2$ for all
sufficiently small $\eps>0$ and $|\lambda|$. If moreover the equation $F(x)=h$ has at most three
real roots for every $h\in(-h_0,h_0)$, then $N_F(\eps,\lambda)\le1$.
\end{proposition}

\begin{proof}
If three consecutive orbits occur for parameters tending to $(0^+,0)$, pass to ordered limits by
Theorem~\ref{thm:limit-sets}. Their stabilities alternate by Lemma~\ref{lem:common-sign},
Proposition~\ref{prop:sdi}, and Lemma~\ref{lem:consecutive}. Choose a regular height in
$(-h_0,h_0)$ away from their jump heights. The three ordered left roots are distinct because
consecutive ones have opposite signs of $F'$. So are the three right roots, and the two groups
are disjoint. This gives six roots of $F(x)=h$, impossible for a quintic. Two consecutive orbits
similarly give four roots, proving the second bound.
\end{proof}

We now determine when non-common cycles are admissible.

\begin{lemma}\label{lem:degenerate-points}
Let $\Gamma$ be an admissible slow--fast cycle that follows an attracting slow arc ending at a
contact point $p$. If $p$ is an inflection point with horizontal tangent, $\Gamma$ continues along
the attracting arc beyond $p$. If $p$ is a quartic extremum, $\Gamma$ leaves $p$ along a fast
segment at the height $F(p)$ and lands at a regular point of an attracting branch.
\end{lemma}

\begin{proof}
At a stationary inflection $\xi\ne0$ the walls $x=\xi\pm c$ point inward in a small height
window. Monotone height drift takes the orbit through it. At a quartic maximum $p<0$ the entry
wall $x=p-c$ stays inward, and for $y\ge F(p)+c^4$ the orbit exits the strip to the right in time
$O_c(1)$. This gives limiting exit height $F(p)$ as $c\downarrow0$, with the minimum case
obtained by reflection. The factorization $F(x)-F(p)=a_5(x-p)^4(x-z)$ directs the layer flow to
the only other root $z$, where $F'(z)>0$. The reversed layer flow leaves on the side without
another root, precluding a compact limiting passage.
\end{proof}

For two quadratic folds $p$ and $q$ of the same type and at the same height,
lying on opposite sides of the origin, the resonance condition in
Theorem~\ref{thm:mono} is
\begin{equation}\label{eq:resonance}
\frac{p^2}{\varkappa_p}=\frac{q^2}{\varkappa_q}.
\end{equation}

\subsection{Fold resonance}\label{sec:res}

We first analyze the resonant configuration needed in the proof of Theorem~\ref{thm:mono}. Figure~\ref{fig:resonance} shows the chain
of this section.

\begin{figure}[tbp]
\centering
\PlotPanel{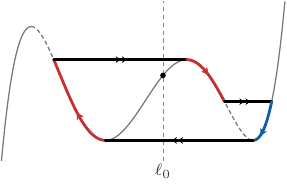}{$\Lambda_1$}{fig:resonance-a}\hfill
\PlotPanel{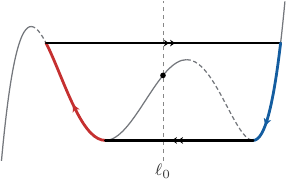}{$\Lambda_2$}{fig:resonance-b}\hfill
\PlotPanel{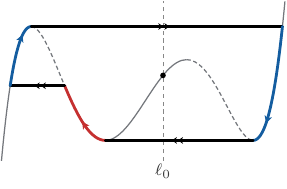}{$\Lambda_3$}{fig:resonance-c}
\par\medskip
\PlotPanel{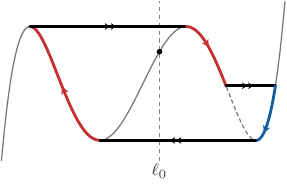}{$\Lambda_1$}{fig:resonance-d}\hspace{.035\textwidth}%
\PlotPanel{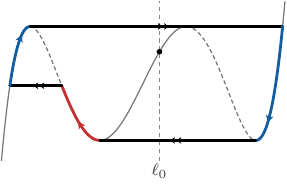}{$\Lambda_3$}{fig:resonance-e}
\caption{The chain $\Lambda_1,\Lambda_2,\Lambda_3$ at a fold resonance, in the normalized
$(X,Y)$-plane. Panels~(a)--(c) show the three families for $0<e<1/\sqrt5$; panels~(d,e) show the
two remaining families for $e=1/\sqrt5$, where $H=H_\alpha=H_\beta$ and $\Lambda_2$ reduces to a
single cycle. The dashed vertical line is $X=\ell_0$, and the black dot is the limiting
equilibrium.}
\label{fig:resonance}
\end{figure}

\subsubsection{Normal form and admissible cycles}

If the equal-height folds are maxima, the reflection $(x,y,\lambda)\mapsto(-x,-y,-\lambda)$,
$F(x)\mapsto-F(-x)$ turns them into minima and preserves both time and plane orientation. Label
the minima $p<0<q$, at the height $h_*$. The remaining root $z$ of $F(x)-h_*$
satisfies $z<p$. Put $x_c=(p+q)/2$, $c_x=(q-p)/2$ and $e=c_x/(x_c-z)$, so that $0<e<1$. With
$c_y=a_5c_x^4(x_c-z)>0$, the change of variables $X=(x-x_c)/c_x$, $Y=(y-h_*)/c_y$ and the time
rescaling by $c_y/c_x>0$ transform system~\eqref{eq:system} into
\[
X'=Y-P(X),\qquad Y'=-\eta(X-\ell),\qquad P(X)=(X^2-1)^2(1+eX),
\]
with $\eta=\eps(c_x/c_y)^2$ and $\ell=(\lambda-x_c)/c_x$. The quadratic coefficients of $P$ at the
two minima are $\varkappa_{\pm1}=P''(\pm1)/2=4(1\pm e)$, so~\eqref{eq:resonance} reads $(\ell_0+1)^2/(1-e)=(1-\ell_0)^2/(1+e)$, where $\ell_0$
is the limit of $\ell$. Hence
\[
\ell\longrightarrow\ell_0=-\varsigma=-\frac{e}{1+\sqrt{1-e^2}},\qquad e=\frac{2\varsigma}{1+\varsigma^2}\quad(0<\varsigma<1).
\]
The other two extrema of $P$ are maxima at
\[
\alpha=\frac{-2-\sqrt{4+5e^2}}{5e}<-1,\qquad \beta=\frac{-2+\sqrt{4+5e^2}}{5e}\in(-1,1),
\]
with heights $H_\alpha=P(\alpha)$ and $H_\beta=P(\beta)$. They separate five inverse branches
$X_1<\alpha<X_2<-1<X_3<\beta<X_4<1<X_5$. Since
$P'(\ell_0)=6\varsigma(1-\varsigma^2)^2/(1+\varsigma^2)>0$, Lemma~\ref{lem:trap} yields a
trapping rectangle around $(\ell_0,P(\ell_0))$ disjoint from all periodic orbits for small parameters.
This excludes $X_3$: the reduced flow on $X_3$ is directed toward $\ell_0$ from both sides, and $P'>0$
throughout the open branch $X_3$. If an admissible limit contained an arc of $X_3$ of positive length,
it could not meet the interior of the trapping rectangle by Lemma~\ref{lem:trap}; picking a regular
point on that arc outside the rectangle and following the forward reduced flow to a regular point
inside the rectangle produces a compact subarc on the normally hyperbolic attracting branch $X_3$.
Nearby periodic orbits entering the attracting tube of this subarc would be carried along the tube into
the interior of the trapping rectangle, contradicting Lemma~\ref{lem:trap}. Hence only $X_1,X_2$
occur on the left and $X_4,X_5$ on the right. The lower intercept is $Y^*_-=0$, because $P\ge0$ on $X\ge-1$ and $Y^*_->0$ would prevent the
leftward fast segment at $X=-1$. In forward time the upper fast segment starts on a left branch
and ends on a right branch, and its horizontal projection contains $\beta$, possibly as its right
endpoint. The jump condition therefore gives $Y^*_+\ge H_\beta$, while the height of the left
endpoint gives $Y^*_+\le H_\alpha$. Since
\[
H_\alpha-H_\beta=\frac{32(1-5e^2)(4+5e^2)^{3/2}}{3125e^4},
\]
there is no admissible cycle, and hence no periodic orbit, if $e>1/\sqrt5$. For $0<e\le1/\sqrt5$,
the jump condition in increasing $Y$ permits only the switches $X_2\to X_1$ and $X_5\to X_4$, each
at most once. If the left half contains an upper arc of positive length on $X_1$, the upper fast
segment starts at the fold $\alpha$, so $Y_+^*=H_\alpha$; if the right half contains an upper arc of
positive length on $X_4$, the same argument in reversed time gives $Y_+^*=H_\beta$. For
$e=1/\sqrt5$ we have $\alpha=-1-2e$, $\beta=1-2e$, $\ell_0=2-5e$ and
$\varkappa_{\alpha,\beta}=4(1\pm e)$, with the plus sign at $\alpha$. The
fold-exit coefficients at $\alpha,\beta$ satisfy
\begin{equation}\label{eq:ratio9}
 \frac{\Delta_L^3}{\Delta_R^3}
 =\frac{(\ell_0-\alpha)^2\varkappa_\beta}{(\beta-\ell_0)^2\varkappa_\alpha}
 =\frac{9(1-e)^3}{(1+e)(3e-1)^2}=9.
\end{equation}
The intervening upper fast segment changes height by $O(\eta)$, so by
Proposition~\ref{prop:fold-exit} the two switches still cannot occur together.
Thus $\preceq$ orders the admissible cycles in the following candidate chain.
\begin{center}\small
\begin{tabular}{@{}llll@{}}
\toprule
Family & Top height & Branches (left; right) & Outward parameter\\
\midrule
$\Lambda_1(\kappa)$ & $H_\beta$ & $X_2$; $X_5\to X_4$ at $\kappa$ & $\kappa$ from $0$ to $H_\beta$\\
$\Lambda_2(h)$ & $h$ & $X_2$; $X_5$ & $h$ from $H_\beta$ to $H_\alpha$\\
$\Lambda_3(h)$ & $H_\alpha$ & $X_2\to X_1$ at $h$; $X_5$ & $h$ from $H_\alpha$ to $0$\\
\bottomrule
\end{tabular}
\end{center}
At $e=1/\sqrt5$, $\Lambda_2$ reduces to the single member at $H_\alpha=H_\beta$.

\subsubsection{Signs of the slow divergence integral}

In the $(X,Y)$-coordinates, $\lim\eta\log\rho=(c_x/c_y)^2\lim\eps\log\rho$, and the slow divergence
integral becomes
\[
I=\int_0^{Y^*_+}\bigl[Q_0(X_-(Y))-Q_0(X_+(Y))\bigr]\,dY,\qquad Q_0(X)=\frac{P'(X)}{X-\ell_0}.
\]

\begin{lemma}\label{lem:res-sign}
For $0<e\le1/\sqrt5$ and every $0<Y<H_\alpha$, we have $Q_0(X_2(Y))<Q_0(X_5(Y))$.
\end{lemma}

\begin{proof}
The branch $X_2$ maps $(0,H_\alpha)$ onto $(\alpha,-1)$, and $X_5$ is defined on
$(0,\infty)$ with values in $(1,\infty)$. Indeed, $P$ decreases from $H_\alpha$ to zero
on the first interval of abscissae and increases from zero to infinity on the second.
Put $u=-1-X_2(Y)>0$ and $v=X_5(Y)-1>0$. Since
$1+e\alpha=(3-\sqrt{4+5e^2})/5>0$, the two square roots
$s_-(u)=\sqrt{1-e-eu}$ and $s_+(v)=\sqrt{1+e+ev}$ are positive, and
\[
 \sqrt Y=u(u+2)s_-(u)=v(v+2)s_+(v).
\]
Using $Q_0=P'/(X+\varsigma)$ and differentiating the two expressions for $\sqrt P$ gives
\begin{align*}
 \frac{Q_0(X_2(Y))}{\sqrt Y}
 &=\frac{4(1+u)s_-(u)-eu(u+2)/s_-(u)}{1-\varsigma+u}
 <\frac{4\sqrt{1-e}}{1-\varsigma},\\
 \frac{Q_0(X_5(Y))}{\sqrt Y}
 &=\frac{4(1+v)s_+(v)+ev(v+2)/s_+(v)}{1+\varsigma+v}
 >\frac{4\sqrt{1+e}}{1+\varsigma}.
\end{align*}
Here $s_-(u)<\sqrt{1-e}$ and $s_+(v)>\sqrt{1+e}$. Since $0<\varsigma<1$,
$(1+u)/(1-\varsigma+u)<1/(1-\varsigma)$ and
$(1+v)/(1+\varsigma+v)>1/(1+\varsigma)$. Finally, resonance gives
$(1-\varsigma)^2/(1-e)=(1+\varsigma)^2/(1+e)$, so the two positive comparison
constants satisfy $4\sqrt{1-e}/(1-\varsigma)=4\sqrt{1+e}/(1+\varsigma)$.
Multiplication by $\sqrt Y>0$ proves the strict inequality.
\end{proof}

\begin{proposition}\label{prop:res}
For $0<e\le1/\sqrt5$, the chain $\Lambda_1,\Lambda_2,\Lambda_3$ contains exactly one neutral cycle,
an interior member $\Lambda_1(\kappa_*)$ with $0<\kappa_*<H_\beta$. The slow divergence integral is
positive on the members inside it and negative on those outside it, and $\Lambda_1(\kappa_*)$ has the
monotone ratio property.
\end{proposition}

\begin{proof}
By Lemma~\ref{lem:res-sign}, integrating $Q_0(X_2(Y))-Q_0(X_5(Y))<0$ from $0$ to $h$
gives $I_2(h)<0$ for every $H_\beta\le h\le H_\alpha$, including $H_\beta=H_\alpha$.
On $\Lambda_1$, $I_1(0)>0$ because both slow
arcs $X_2$ and $X_4$ are repelling, $I_1(H_\beta)=I_2(H_\beta)<0$, and
$I_1'(\kappa)=Q_0(X_4(\kappa))-Q_0(X_5(\kappa))<0$. On $\Lambda_3$,
$I_3'(h)=Q_0(X_2(h))-Q_0(X_1(h))>0$, so $I_3$ decreases along the chain from
$I_3(H_\alpha)=I_2(H_\alpha)<0$. This proves the first two assertions.

We shoot from the regular moving jump with $\theta=\kappa$. The first captures are on $X_5$ in forward time
and $X_4$ in backward time, with $Q_0(X_4)-Q_0(X_5)=I_1'(\kappa)<0$ in~\eqref{eq:shooting-ratio}. We set
$\nu=\eta$ and $G(W)=P(\ell\pm\sqrt{2W})$ for the transition estimates. After a further compact
arc on $X_5$, Lemma~\ref{lem:corridor} gives polynomial bounds past $1$ to $\ell_0<a_+<1$. The
lower jump from $a_+$ to $\ell$ at height $o(1)$ is regular since $P>0$ there.

In backward time, we use a compact arc on $X_4$ and the same lemma past $\beta$ to $\ell_0<a_\beta<\beta$.
The upper jump to a cut $z_b$ just right of $X_2(H_\beta)$, with $z_b<-1$, is regular since
$H_\beta-P>0$ on $[z_b,a_\beta]$. Choose $d>0$ with $P(z_b)<H_\beta-2d$ and put $b_2=H_\beta-d$.
If $H_\beta<H_\alpha$, regular capture on $X_2$ gives polynomial bounds at $b_2$. Otherwise
$X_2(H_\beta)=\alpha$ and the whole segment remains at $X\ge\alpha-\gamma$ by
Lemma~\ref{lem:shooting-annulus}. The reversed field points into this receiving strip at $z_b$
down to $b_2$, and $G'\ge-K\gamma$, $|G''|\le K$ there. Formula~\eqref{eq:receiving-strip} gives
$c_-=2dK\gamma<\chi_-$ for fixed sufficiently small $\gamma$. From the regular tube at $b_2$ a
compact arc of $X_2$ and Lemma~\ref{lem:corridor} give polynomial bounds past $-1$ to
$-1<a_-<\ell_0$. The lower fast map is regular up to the common terminal half-line $\{X=\ell,\
Y<P(\ell)\}$. All other exponents vanish, so Lemma~\ref{lem:composition} gives the property.
\end{proof}

\begin{proof}[Proof of Theorem~\ref{thm:mono}]
If $F$ has no extremum, Bendixson's criterion for $-F'\le0$ gives $N_F=0$. In a mixed cycle a
transition from attraction to repulsion cannot leave a regular attracting arc, enter a regular
repelling arc, or use a degenerate contact by Lemma~\ref{lem:degenerate-points}. Nor do both slow
arcs join directly at a remote quadratic fold, since their heights lie on the same side of the
fold height and the height is monotone locally. Thus the transition jumps at height $H$ from a
quadratic fold $p$ in forward time to a quadratic fold $q$ in backward time. The jump condition
excludes an intermediate simple root. Hence $F(x)-H=a_5(x-p)^2(x-q)^2(x-z)$ with $z$ outside the
interval between the folds. They have the same type, and Rolle's theorem gives four quadratic
folds. The slow directions force $p<0<q$ at maxima and $q<0<p$ at minima. If $F$ has at most two
extrema, all limits are therefore common and Proposition~\ref{prop:common} gives $N_F\le1$ since
there are at most three roots at each height.

For the equal-height pair $p,q$, put $H=F(p)=F(q)$ and choose
two fixed vertical sections inside the fast segment, one near
each fold. Between the sections, $H-F(x)$ is bounded away from
zero. The same holds for $y-F(x)$ along the nearby connecting
orbit, so
\[
\frac{dy}{dx}=-\eps\frac{x-\lambda}{y-F(x)}=O(\eps).
\]
Hence the heights $y_p$ and $y_q$ at the two sections differ
by $O(\eps)$.

The limiting orbit follows a slow arc of fixed positive length
into $p$ in forward time and into $q$ in backward time.
Proposition~\ref{prop:fold-exit} therefore gives
\[
y_j=H+\iota\zeta
\left(\frac{|j-\lambda|^2}{\varkappa_j}\right)^{1/3}
\eps^{2/3}+o(\eps^{2/3}),\qquad j\in\{p,q\}.
\]
The sign $\iota$ is the same at both exits because the folds
have the same type. Subtracting the two expansions, dividing
by $\eps^{2/3}$ and letting $(\eps,\lambda)\to(0^+,0)$ gives
the resonance condition~\eqref{eq:resonance}.
If~\eqref{eq:resonance} fails for both pairs of extrema of the same type at equal height, every
admissible cycle is common, and Proposition~\ref{prop:common} gives $N_F\le2$.

It remains to consider the resonant case.
For $e>1/\sqrt5$ the admissible set is empty. Otherwise split it
by its position relative to $\Lambda_*=\Lambda_1(\kappa_*)$, putting
$\mathcal A_0=\mathcal A\cap\{\Lambda_*\}$. Proposition~\ref{prop:res}
gives the signs and monotone ratio property, and the chain gives the order.
\end{proof}
\section{Proof of Theorem~\ref{thm:main}}\label{sec:final}

We first count geometrically distinct periodic orbits.
Suppose that three consecutive cycles occur for parameters tending to $(0^+,0)$. By
Lemma~\ref{lem:annulus} we can choose a consecutive triple whenever there are at least three.
Pass to ordered limits by Theorem~\ref{thm:limit-sets}. The alternatives with direct bounds in
Theorems~\ref{thm:minimum} and~\ref{thm:mono} already exclude a triple. For a local maximum use
the normalization of Section~\ref{sec:min-setup}.

In the remaining cases assign the types $\mathsf S,\mathsf I,\mathsf N,\mathsf O$ to the
equilibrium, the inner limits, the neutral limit, and the outer limits, respectively. They occur
in that order, allowing repetitions, by the order relations of the classification and
antisymmetry in Theorem~\ref{thm:limit-sets}. Type $\mathsf S$ occurs only in
Theorem~\ref{thm:minimum}(ii), by Lemma~\ref{lem:trap}. Types $\mathsf S,\mathsf I$ are
hyperbolic repelling for large $n$ by Propositions~\ref{prop:shrinking} and~\ref{prop:sdi}, and
type $\mathsf O$ is hyperbolic attracting by Proposition~\ref{prop:sdi}. By
Theorem~\ref{thm:neutral}, the neutral block has at most two members, a repeller followed by an
attractor when there are two. A nonhyperbolic neutral member has multiplicity two and pattern
$(-,-)$.

Consequently the patterns in any ordered consecutive block are repellers $(-,+)$ followed by
attractors $(+,-)$, with a possible single $(-,-)$ between them. Lemma~\ref{lem:consecutive}
excludes two consecutive repellers or two consecutive attractors. It also excludes either
neighbor of a $(-,-)$ member, since the inner neighbor would have outer sign $+$ and the outer
neighbor would have inner sign $+$. Thus a consecutive triple is impossible.

It remains to count multiplicities. Suppose that the bound counted
with multiplicity fails along a sequence of parameters tending to
$(0^+,0)$. By the geometric bound just proved, we may choose a
nonhyperbolic periodic orbit at each parameter pair such that this
orbit either has multiplicity greater than two or coexists with
another periodic orbit. Pass to a subsequence on which the chosen
orbits have a Hausdorff limit.

By Proposition~\ref{prop:shrinking}, the limit cannot be the
equilibrium. In Theorem~\ref{thm:minimum}(i), all admissible
slow--fast cycles have $I<0$. In Theorem~\ref{thm:mono}(i),(ii),
they are common and have $I\ne0$ by Lemma~\ref{lem:common-sign}.
Proposition~\ref{prop:sdi} therefore leaves only a limit of
type $\mathsf N$. Applying Theorem~\ref{thm:neutral}(c) with
coincident bounding orbits, we find that, for all sufficiently
large indices, the chosen orbit has multiplicity exactly two
and sign pattern $(-,-)$.

A second orbit must therefore coexist with it. By the geometric
bound the two orbits are consecutive. Pass to a further subsequence
on which the second orbit lies on a fixed side of the first and
has a Hausdorff limit. If both limits are neutral, they coincide,
and Theorem~\ref{thm:neutral}(b) makes both orbits hyperbolic,
a contradiction. Otherwise, the ordering and stability established
above imply that an inner neighbor is repelling and an outer
neighbor is attracting. In either case its displacement is positive
on the side facing the nonhyperbolic orbit, whose displacement is
negative on that side. This contradicts Lemma~\ref{lem:consecutive}
and proves the bound counted with multiplicity. \qed

\appendix
\renewcommand{\thetheorem}{\Alph{section}.\arabic{theorem}}
\renewcommand{\theequation}{\Alph{section}.\arabic{equation}}

\section{Scale analysis for shrinking cycles}\label{app:shrinking}

Write $a$ for the lowest coefficient and $d=a_5$, so $F_{(1)}=ax^2+cx^4+dx^5$ and
$F_{(2)}=ax^4+dx^5$, with $a,d>0$. We consider sequences with $(\eps,\lambda)\to(0^+,0)$ and
$\varrho\to0$, using $\delta=\eps^{1/(4m-2)}$ and $R=\varrho/\delta$.

\subsection{An identity and a bound on the parameter}\label{sec:shrink-id}

Comparing the two half-orbits at the same height gives a moment identity
and a formula for the return multiplier.

\begin{lemma}\label{lem:shrinking-identities}
For $X'=Y-P(X)$, $Y'=-X$ with $P(X)=\sum_{j=1}^5p_jX^j$, consider a periodic orbit with vertical
intercepts $Y_-<0<Y_+$, and let $w_\pm=X_\pm^2/2$ on its halves. Put
\begin{gather*}
D=w_+-w_-,\qquad S_j=(2w_+)^{j/2}+(2w_-)^{j/2},\\
\Theta=2p_2+4p_4(w_++w_-),\qquad \omega'=-\Theta\omega\quad(\omega>0),
\end{gather*}
$M_j=\int_{Y_-}^{Y_+}\omega S_j\,dY$ and $G=p_1S_1+p_3S_3+p_5S_5$. Then
\begin{gather*}
D'=\Theta D+G,\qquad p_1M_1+p_3M_3+p_5M_5=0,\qquad
D(Y)=-\int_Y^{Y_+}e^{-\int_Y^v\Theta}G(v)\,dv,\\
\log\rho=-\int_{Y_-}^{Y_+}\bigl[8p_4D+p_1S_{-1}+3p_3S_1+5p_5S_3\bigr]\,dY .
\end{gather*}
\end{lemma}

\begin{proof}
The two equations for $w_\pm$ give the difference equation. Multiplication by $\omega$ and the
zero endpoints give the moment and integral formulas. The multiplier identity follows by
integrating $-P'(X)\,dt$ on both halves with $dt=-dY/X$. The endpoint singularities are
integrable by transverse crossing.
\end{proof}

In the original coordinates centered at $x=\lambda$, the odd part divided by $z=x-\lambda$ is
\[
\frac{\Fo(W)}{z}=
\begin{cases}
F_{(1)}'(\lambda)+(4c\lambda+10d\lambda^2)z^2+dz^4, & m=1,\\
F_{(2)}'(\lambda)+(4a\lambda+10d\lambda^2)z^2+dz^4, & m=2.
\end{cases}
\]
For small $\lambda\ge0$ this is positive for $z>0$: for $m=1$ and $\lambda>0$ because
$F_{(1)}'(\lambda)-(4c\lambda+10d\lambda^2)^2/(4d)>0$, and for $m=2$ because all coefficients are
nonnegative. Hence Lemma~\ref{lem:odd} excludes periodic orbits when $\lambda\ge0$. For $\lambda<0$,
a periodic orbit requires a zero of this quotient at some $z\in(0,\varrho]$, which gives
\begin{equation}\label{eq:shrinking-parameter-bound}
|\lambda|=O\bigl(\varrho^{6-2m}\bigr).
\end{equation}
For $m=1$, $c_0|\lambda|\le K|\lambda|\varrho^2+d\varrho^4$ with $c_0>0$ gives the bound.
For $m=2$ put $b=-\lambda>0$. At the zero,
$dz^4=(4a-10db)bz^2+(4a-5db)b^3\ge2abz^2$, hence $b=O(\varrho^2)$.

\subsection{Amplitudes below the inner scale (\texorpdfstring{$R\to0$}{R -> 0})}\label{sec:shrink-weak}

For $R\to0$, use $X=(x-\lambda)/\varrho$, $Y=(y-F_{(m)}(\lambda))/(\varrho\sqrt\eps)$ and time
$\vartheta=\sqrt\eps\,t$. The rescaled polynomial has even part $O_{C^2}(R^{2m-1})$ on bounded
intervals and odd part $\nu[dX^5-2ma\hat\lambda X^r+o_{C^2}(1)]$, where
$\nu=\varrho^4/\sqrt\eps\to0$, $r=2m-1$ and $\hat\lambda=-\lambda/\varrho^{6-2m}$ is bounded
by~\eqref{eq:shrinking-parameter-bound}. At $|X|=1$ the orbit has $X'=0$ and $Y=o(1)$. Smooth
dependence from this point and transverse crossings give convergence to the unit circle over one
turn and convergence of the period to $2\pi$.

In Lemma~\ref{lem:shrinking-identities}, $\omega=1+o(1)$ uniformly with $\omega(0)=1$, and
$D=O(\nu)$ by its integral formula. For $j\ge-1$, $\int S_j\,dY$ equals $\int
|X|^{j+1}\,d\vartheta$ over a period. The same identity with $\omega(Y)$ gives
$M_j\to\kappa_{j+1}$, where $\kappa_k=\int_0^{2\pi}\cos^k t\,dt$ for the even indices used here,
including $\kappa_0=2\pi$. The moment balance gives $p_r/\nu\to-d\kappa_6/\kappa_{r+1}$, since
the other lower odd coefficient is $o(\nu)$. With $p_4=o(1)$, the contribution of $8p_4D$ is
$o(\nu)$. The multiplier identity and $\kappa_{j+2}=(j+1)\kappa_j/(j+2)$ now give
\[
 \frac{\log\rho}{\nu}
 =d\left(\frac{r\kappa_{r-1}\kappa_6}{\kappa_{r+1}}-5\kappa_4\right)+o(1)
 =-(6-2m)d\kappa_6+o(1)<0 .
\]
\subsection{Amplitudes of the order of the inner scale (\texorpdfstring{$R\to R_*\in(0,\infty)$}{R -> R*})}\label{sec:shrink-finite}

When $R$ tends to a finite positive limit, the sign of the leading term
in $-\log\rho$ is determined by an integral involving the limiting
half-orbit profile. The following lemma shows that this integral is positive.

\begin{lemma}\label{lem:moment-ratio}
Let $u$ be positive and strictly concave on $(L,H)$, of class $C^2$ with simple endpoint zeros. For
a positive weight function $\omega$ on $(L,H)$, $p,q>0$, and a measurable function $\beta$ on $(L,H)$
satisfying $0\le\beta(t)\le p+q$, put $T=\int_L^H\omega u^{p+q}\,dt/\int_L^H\omega u^p\,dt$,
assuming the moments are finite. Then
\[
 \int_L^H u^{p-1}\bigl[(p+q-\beta(t))u^q-(p-\beta(t))T\bigr]\,dt>0 .
\]
\end{lemma}

\begin{proof}
Where $u^q>T$, the bracket is $(p+q-\beta)(u^q-T)+qT>0$.
On $u^q\le T$ the term $-\beta(u^q-T)$ is nonnegative. By concavity the level
density $\omega_0(v)=1/u'(t_-(v))-1/u'(t_+(v))$ is increasing, where the two
$t_\pm(v)$ are the left and right solutions of $u(t)=v$. Put $z=T^{1/q}$.
Since $\int_0^z v^{p-1}[(p+q)v^q-pT]\,dv=0$, subtracting the density at
$v=(pT/(p+q))^{1/q}$ gives a nonnegative integral on $u^q\le T$.
The region $u^q>T$ has positive length since $T<(\max u)^q$.
\end{proof}

Use $X=(x-\lambda)/\delta$, $Y=(y-F_{(m)}(\lambda))/\delta^{2m}$ and time
$\delta^{2m-1}t$. The rescaled polynomial converges to $aX^{2m}$. For $H=Y_+$,
the point with $|X|=R$ gives $H\ge P(\pm R)=aR^{2m}+o(1)$.
If $|P|\le K$ on $|X|\le R$, integrating the two outer parts of
$w_\pm'=P(\pm\sqrt{2w_\pm})-Y$ gives $H\le K+R$ and $Y_-\ge-K-R$.
Thus the upper intercepts range over a compact subinterval of $(0,\infty)$.

For $m=1$ the limiting half-profile is $u=Y/(2a)+1/(4a^2)-he^{2aY}$, with
$h=(H/(2a)+1/(4a^2))e^{-2aH}\in(0,1/(4a^2))$. It is strictly concave,
positive at zero and negative for large negative $Y$, and its lower
zero is simple. For $m=2$, multiply $X$ and $Y$ by $(4a)^{1/3}$ without
changing time. The limiting even polynomial becomes $X^4/4$,
and the half-profile $u=X^2/2$ satisfies $u'=u^2-Y$, $u(H)=0$ in the new height.
Backward comparison gives $0<u\le\sqrt H$ for $0<Y<H$.
A zero at $Y=0$ would have $u'=0$ and $u''=-1$, so $u(0)>0$.
For $Y<0$ while $u>0$, $u(Y)\le u(0)-Y^2/2$; hence it has a lower
zero $L<0$ with $u'(L)=-L>0$. At both endpoints $u''=2uu'-1=-1$.
At any zero of $u''$, $u'''=2(u')^2>0$, forcing an upward crossing.
The endpoint signs therefore give $u''<0$ throughout.

Both profiles have transverse crossings in the planar flow at the upper
and lower intercepts. Compactness of the positive range of upper intercepts
and smooth dependence give uniformly bounded flight times and crossing
speeds bounded away from zero for the limiting family and its small perturbations.

For the quartic profile, write the height as $t$ and set
\[
 f(t)=\exp\Bigl(-\int_L^tu(v)\,dv\Bigr),\qquad
 g(t)=f(t)\int_L^tf(s)^{-2}\,ds,\qquad
 K(t,s)=f(t)f(s)\int_s^tf(v)^{-2}\,dv .
\]
Then $f''=tf$, $g''=tg$ and $fg'-f'g=1$. The function $g'$ solves
$(g')''=tg'+g$ with $g'(L)=1$, $g''(L)=0$. Variation of constants gives
$g'(t)=f(t)+\int_L^t K(t,s)g(s)\,ds>0$. Thus
$\beta=2ufg=2-2fg'$ belongs to $(0,2)$ inside the interval and vanishes
at the endpoints.

Write $\sigma_\delta=\delta^{5-2m}$ and $\hat\mu=\lambda/\delta^{6-2m}$.
Pass to a subsequence on which the upper intercept and the bounded
$\hat\mu$ converge. In the inner coordinates before the extra quartic
scaling, Taylor expansion gives odd part divided by $\sigma_\delta$
equal to $2ma\hat\mu X^{2m-1}+dX^5+o_{C^2}(1)$ on bounded sets.
In particular for $m=2$, $p_1=O(\delta^3)$ also after the fixed scaling.
Using height $t$ in either normalization, the limiting odd part evaluated at
$X=\sqrt{2u(t)}$ is $q_0(u)=A_0u^p+B_0u^{p+q}$ with
$p=(2m-1)/2$, $q=(6-2m)/2$, and $B_0>0$.

By transverse smooth dependence, the half-profiles converge at every
interior height and the intercepts and flight times converge. Their slopes
are uniformly bounded, so the zero extensions converge uniformly on a
common compact height interval. In Lemma~\ref{lem:shrinking-identities},
$\Theta$ and $G/\sigma_\delta$ are bounded and converge on the limiting
interior, with $G/\sigma_\delta\to2q_0(u)$. The moment identity gives
$A_0=-B_0T$, with the limiting weight $\omega$ solving $\omega'=-2a\omega$
for $m=1$ and $\omega=f^2$ for $m=2$. In the multiplier integral,
$\int S_j\,dt$ equals the time integral of $|X|^{j+1}$ on the cycle,
so the terms with $j\ge-1$ converge, including the period at $j=-1$.
For $m=2$ the term with $p_1 S_{-1}$ is $O(\delta^3)=o(\sigma_\delta)$.
The surviving odd terms in $-\log\rho/\sigma_\delta$ thus give
$2\int_L^H q_0'(u)\,dt$.

For $m=1$, $p_4=o(1)$ and $D=O(\sigma_\delta)$, so the even term
vanishes in this limit. For $m=2$, $8p_4\to2$. The integral formula
for $D$ and bounded convergence on a compact height square give
\[
 \int \frac{D}{\sigma_\delta}\,dt\longrightarrow
 -2\int_L^H\int_t^H\frac{f(v)^2}{f(t)^2}\,q_0(u(v))\,dv\,dt
 =-2\int_L^H f(v)g(v)q_0(u(v))\,dv .
\]
The interchange uses Fubini on a bounded integrand, since $f>0$
on the compact interval. With $\beta=0$ in the first case and
$\beta=2ufg$ in the second, we obtain
\[
 \lim\frac{-\log\rho}{\sigma_\delta}
 =2B_0\int_L^H u^{p-1}\bigl[(p+q-\beta)u^q-(p-\beta)T\bigr]\,dt>0
\]
by Lemma~\ref{lem:moment-ratio}. The same argument applies to any
subsequence with $\rho\ge1$, which proves eventual hyperbolic attraction.

\subsection{Amplitudes above the inner scale (\texorpdfstring{$R\to\infty$}{R -> infinity})}\label{sec:shrink-large}

Assume $R\to\infty$. In the inner coordinates $X=(x-\lambda)/\delta$,
$Y=(y-F_{(m)}(\lambda))/\delta^{2m}$ and the time $\delta^{2m-1}t$, put
\begin{gather*}
b=-\frac\lambda\delta,\qquad H=Y_+,\qquad L=Y_-,\\
P(X)=\frac{F_{(m)}(\lambda+\delta X)-F_{(m)}(\lambda)}{\delta^{2m}}=\sum_{j=1}^5p_jX^j,\qquad
P_\pm(w)=P(\pm\sqrt{2w}) .
\end{gather*}
By the beginning of this appendix $\lambda<0$, so $b>0$. We have $\delta R=\varrho\to0$ and
$b/R=|\lambda|/\varrho\to0$ by~\eqref{eq:shrinking-parameter-bound}; hence $R^{-2m}P(Ru)\to au^{2m}$
uniformly for $|u|\le1$. At $|X|=R$ the height equals $P(\pm R)$. Moreover, with
$M=\max_{|X|\le R}P(X)$, integrating $-w'\ge Y-M$ from $M$ to $H$ when $H>M$ gives $H\le M+R$. It
follows that $H\sim aR^{2m}$.

For small $r\ge0$ we have $\frac{d}{dr}F_{(m)}(-r)\ge kr^{2m-1}$
with $k>0$ independent of the parameters. Integrating from $\delta b$ to $\delta(b+\sqrt{2w})$ gives
$P_-(w)\ge k_1w^m$ on the whole left half-orbit. We use the following comparison twice. Suppose
that $P_\pm(w)\ge kw^m-C_0$ on $Y\ge0$. For either $m$, choose $K>0$ with $kK^m>1$ and $M_0\ge1$
with $kK^mM_0>C_0+K/m$, and set $B(Y)=K(M_0+Y)^{1/m}$. Then $B'(Y)=(K/m)(M_0+Y)^{1/m-1}\le K/m$,
so at a contact $kB^m-C_0-Y-B'\ge(kK^m-1)Y+kK^mM_0-C_0-K/m>0$. Since $w(H)=0<B(H)$, backward
comparison gives $w(Y)\le B(Y)$ for $0\le Y\le H$. With $C_0=0$ this bounds $w_-(0)$. On
$[L,0]$ we have $w_-'\ge-Y$, whence $L^2/2\le w_-(0)\le B(0)$. On the other hand, the orbit must enter
$\{x>0\}$: otherwise it would lie in the small interval of abscissae $x<0$ on which the divergence
$-F_{(m)}'(x)$ is positive, which contradicts Bendixson's criterion. At its leftward crossing of
$x=0$ we have $\dot x=y<0$ (a tangency there would have $\ddot x=\eps\lambda<0$ and could not be a
crossing), and $y$ decreases further until the orbit reaches $x=\lambda$. Consequently
$L<-F_{(m)}(\lambda)/\delta^{2m}\le-k_2b^{2m}$. Together with the bound on $L$, this gives
$b\le K_0$ for both values of $m$.

Expanding $P$ and using $\delta(R+b)\to0$ with $b$ bounded, we obtain
on $0\le w\le R^2/2$
\[
k_3w^m-K_1\le P_\pm(w)\le K_1(1+w^m).
\]
The same comparison, now with $k=k_3$ and $C_0=K_1$, gives
$w_\pm(Y)\le K(M_0+Y)^{1/m}\le K'(1+Y)^{1/m}$ for $Y\ge0$,
and on $[L,0]$ the inequality $w_\pm'\ge-K_1-Y$ gives a uniform upper bound. For the lower bounds put
$\alpha=1-1/m$ and $T=H-H^{-\alpha}$. On $[T,H]$ set $\tau=H^\alpha(H-Y)$ and $z=w/H^{1/m}$; then
$dz/d\tau\ge1/2-K_2z^m$ for $0\le\tau\le1$ and large $H$, and $k_4\tau$ is a lower solution for a
small fixed $k_4>0$. Thus $w_\pm(Y)\ge k_4H(H-Y)$ on $[T,H]$. At a contact with $k_5Y^{1/m}$ below
$T$ one has $w'-(k_5Y^{1/m})'<0$ for a sufficiently small $k_5$ and all $Y\ge Y_0$, with $Y_0$ fixed
and large; backward comparison from $T$ gives $w_\pm(Y)\ge k_5Y^{1/m}$ on $[Y_0,T]$.

Normalize $\omega(0)=1$ in Lemma~\ref{lem:shrinking-identities}. For
$m=1$ the coefficient expansion gives $\Theta=2a+o(1)$ on the whole orbit, since
$\delta^2R^2\to0$. For $m=2$ it gives $\Theta\ge0$ everywhere and $\Theta\ge k_6\sqrt Y$ on
$[Y_0,T]$. Thus, for $Y\ge Y_0$, the weight is bounded above by $Ke^{-k_7Y}$ and by
$Ke^{-k_7Y^{3/2}}$, respectively; in the quartic case this bound extends from $T$ to $H$ because
$H^{3/2}-T^{3/2}=O(1)$. On each fixed bounded height interval $\Theta$ is bounded, so the weight is
bounded above and below there. The upper bounds for $w_\pm$ give $M_1,M_3,M_5\le K$, and the lower
bound on $[Y_0,Y_0+1]$ gives $M_1,M_3\ge k_8>0$.

For $m=1$ the coefficients satisfy $-p_1\ge k_9b$, $|p_3|\le K\delta^2b$ and
$p_5=d\delta^3$, and the identity $p_1M_1+p_3M_3+p_5M_5=0$ of
Lemma~\ref{lem:shrinking-identities} yields $k_{10}b\le K\delta^2b+K\delta^3$, hence
$b=O(\delta^3)$. For $m=2$ we have $p_1\le0$, $-p_3\ge k_{11}b$ and $p_5=d\delta$, and the same
identity gives $k_{11}bM_3\le d\delta M_5$, hence $b=O(\delta)$.

After these bounds on $b$, the coefficient expansions give $P_\pm(w)=Aw^m+o(1)(1+w^m)$ on the
orbit, where $A=2^ma$. On bounded height intervals the functions $w_\pm$ and their first derivatives
are bounded. Along any convergent sequence of lower endpoints, the limits of the profiles are
nonnegative solutions of $w'=Aw^m-Y$ with $w(L)=0$; this follows by passing to the limit in the
integral equation from the moving endpoint. The lower endpoint cannot tend to zero, because the
solution with $w(0)=0$ has $w'(0)=0$ and $w''(0)=-1$; thus $L\le-k<0$. Near $L$, the uniform bound on
$w'$ gives $w(L+v)\le Kv$, and the differential equation then gives $w'(L+v)\ge k/2$ for all
sufficiently small $v\ge0$. Beyond this fixed endpoint interval and up to $Y_0$, the profiles are
bounded away from zero: otherwise a limiting nonnegative profile would have an interior zero, which
would force $Y=0$ and $w''=-1$ there. At $Y_0$ the lower comparison above applies. Together with the
bounds on $[Y_0,T]$ and $[T,H]$, this proves
\[
\int_L^HS_{-1}\,dY=O\bigl(H^{1-1/(2m)}\bigr),\qquad \int_L^HS_1\,dY=O\bigl(H^{1+1/(2m)}\bigr).
\]

For $W(s)=H^{-1/m}w(Hs)$, $\eta=H^{1/m-2}$ and $\phi=(s/A)^{1/m}$, one has $\eta W'=AW^m-s+o(1)$.
On compact intervals inside $(0,1)$, $(W-\phi)'$ has the sign of $W-\phi$ with magnitude at least
$c_\zeta/\eta$ outside $|W-\phi|\le\zeta$. In decreasing $s$ the bounded solutions enter this
strip within $O(\eta)$ and stay there. Thus $W\to\phi$ locally uniformly, with the upper profile
bound providing an integrable majorant for positive moments.

For $m=1$ the bound on $b$ gives $p_1=O(\delta^3)$, $p_3=O(\delta^5)$, $p_4=O(\delta^2)$ and
$p_5=d\delta^3$. Since $\Theta\ge a$, the integral formula for $D$ in
Lemma~\ref{lem:shrinking-identities} gives $|D|\le K\delta^3H^{5/2}$. The profile limit gives
\[
\int_L^HS_3\,dY\sim2a^{-3/2}H^{5/2}\int_0^1s^{3/2}\,ds=\tfrac45a^{-3/2}H^{5/2}.
\]
Relative to $\delta^3H^{5/2}$, the integrals of $8p_4D$, $p_1S_{-1}$ and $3p_3S_1$ are
$O(\delta^2H)$, $O(H^{-2})$ and $O(\delta^2/H)$, respectively, and all three tend to zero, since
$\delta^2H\sim a\varrho^2$. Thus $\log\rho\sim-4da^{-3/2}\delta^3H^{5/2}\sim-4ad\,\varrho^5/\eps$,
which is~\eqref{eq:mult-1}. In particular $\log\rho<0$, also when $\varrho^5/\eps\to0$.

For $m=2$ we have $p_1=O(\delta^3)$, $p_2=O(\delta^2)$, $p_3=O(\delta)$, $p_4=a+O(\delta^2)$ and
$p_5=d\delta$. In this case the term $8p_4D$ is of the same order as $5p_5S_3$ and has the opposite
sign, so $D$ has to be computed on the whole height interval, including the thin upper layer
$[T,H]$ with $T=H-H^{-1/2}$, where $\Theta$ is not bounded below by a multiple of $\sqrt Y$.

On $[T,H]$ the integral formula for $D$, together with $\Theta\ge0$ and $|G|\le K\delta H^{5/4}$,
gives $|D|\le K\delta H^{3/4}$. On $[Y_0,T]$ we have $|G(v)|/\Theta(v)\le K\delta v^{3/4}$ and
\[
\int_Y^Te^{-\int_Y^v\Theta}\,\Theta(v)\,dv\le1 ;
\]
splitting the integral formula at $T$ gives the same bound for $D$ on $[Y_0,T]$. On $[L,Y_0]$ the
interval is bounded and $G=O(\delta)$, so $|D|\le K\delta H^{3/4}$ throughout.

Put $d_H(s)=D(Hs)/(\delta H^{3/4})$, $a_H(s)=\Theta(Hs)/\sqrt H$ and $g_H(s)=G(Hs)/(\delta
H^{5/4})$. Then $H^{-3/2}d_H'=a_Hd_H+g_H$. On $[\eta,1-\eta]$ for fixed $0<\eta<1/2$, the profile
limits give $a_H\to a_0=4\sqrt{as}$ and $g_H\to 2d(s/a)^{5/4}$ uniformly. Multiplying by
$a_0^{-1}$ and integrating, the derivative term is $O_\eta(H^{-3/2})$ by integration by parts and
the global bound on $d_H$. The ratio $a_H/a_0$ tends uniformly to one there. The omitted ends
contribute $O(\eta+|L|/H)$ to the integral of $d_H$. Letting $\eta\downarrow0$ and integrating
the limiting density $-(d/2)a^{-7/4}s^{3/4}$, we obtain, with the profile estimate for $S_3$,
\[
\int_L^HD\,dY\sim-\frac{2d}7a^{-7/4}\delta H^{7/4},\qquad
\int_L^HS_3\,dY\sim\frac87a^{-3/4}H^{7/4}.
\]
In units of $da^{-3/4}\delta H^{7/4}$, the terms $8p_4D$ and $5p_5S_3$ in $-\log\rho$ therefore
contribute $-16/7$ and $40/7$, respectively; the two remaining terms are of relative orders
$O(\delta^2/H)$ and $O(H^{-1/2})$. Using $H\sim aR^4$ and $\delta^6=\eps$, we obtain
\[
\log\rho=-\frac{24ad}7\,\frac{\varrho^7}\eps\,(1+o(1)),
\]
which is~\eqref{eq:mult-2} and completes the proof of Proposition~\ref{prop:shrinking-scales}.
\section{Algebraic results of Li and Lu}\label{app:lilu}

We use the proof of~\cite[Lemmas~3.2 and~3.3]{LiLu}, with attention to
the boundary configurations.

Let $p_-<p_+$ be the outermost critical points of $F$, and put $\Delta=p_+-p_-$, $k_0=5a_5\Delta^4$,
$X=(x-p_-)/\Delta$ and $X_0=-p_-/\Delta$. With
\[
\Phi(X)=\frac{F(p_-)-F(p_-+\Delta X)}{\Delta k_0},\qquad \bar h=\frac{F(p_-)-h}{\Delta k_0},
\]
the derivative $\Phi'$ has leading coefficient $-1$ and outermost zeros $X=0$ and $X=1$, and
\begin{equation}\label{eq:LL-normal}
\frac{F'(x)}{x}=-\frac{k_0}{\Delta}\,\frac{\Phi'(X)}{X-X_0},\qquad
\frac{F'(x)^2}{x}\,dx=k_0^2\,\frac{\Phi'(X)^2}{X-X_0}\,dX .
\end{equation}
The time change $\tau=k_0t$ preserves orientation, so $I(h)=k_0^2\,I_\Phi(\bar h)$ at corresponding
heights, where $I_\Phi$ is the slow divergence integral in the normalization of~\cite{LiLu}. The
factor $k_0^2>0$ and the affine change $h\mapsto\bar h$ preserve signs and the multiplicities of zeros.
For $p_-=-s$ and $p_+=t$ the inner folds become $a_{\mathrm{L}}=s/\Delta$ and $b_{\mathrm{L}}=(1+s)/\Delta$, with heights $A/(\Delta k_0)$ and $(A-B)/(\Delta k_0)$. The fold at $1$ in the
$X$ coordinate has height $(A-C)/(\Delta k_0)$. These heights identify the regions of
\cite[Figure~11]{LiLu}. For the left-end scaling used below, let
$X=-x/p_*$ and $\Phi(X)=F(-p_*X)/(5a_5p_*^5)$, where $-p_*$ is the
leftmost extremum. Then $\Phi'(X)/X=Q(-p_*X)/(5a_5p_*^3)$.

\begin{proposition}\label{prop:LL}
\textup{(a)} In the two-extrema configuration with a nondegenerate minimum
at $0$, or in the four-fold configuration $-s<0<1<t$, $J'$ has at most
one zero in $(0,h_1)$, where $h_1=A$ or $\min(A,B)$, respectively.

\textup{(b)} If $A>B>C>0$ in the four-fold configuration, $I_4$ has at most
one zero in $(B,A)$ and any such zero is simple.

\textup{(c)} Let $F'=5a_5x(x+p_1)(x+p_2)(x+p_3)$ with $0<p_1<p_2<p_3$
and $F(-p_2)\ge0$. At equal positive heights, $Q(x^-)<Q(x^+)$ when
$x^-\in(-p_3,-p_2)$ and $x^+>0$.

\textup{(d)} For $F'=x(x+2)(x-3)^2$, $J'$ has no zero in $(0,A)$,
where $A=F(-2)=F(3)$.
\end{proposition}

\begin{proof}
For (a) with two extrema, the left-end scaling has $p_*=s$ and
$\Phi'=-X(X^2-\alpha_{\mathrm{L}}X+\beta_{\mathrm{L}})(X-1)$ as in
\cite[(3.4)]{LiLu}. The branches are $\sigma_1(X)\in(0,1)$ and $X<0$.
For $\beta_{\mathrm{L}}\ge(1-\alpha_{\mathrm{L}})^2/4$, formula~(3.13) gives a negative divided difference.
For $\beta_{\mathrm{L}}<(1-\alpha_{\mathrm{L}})^2/4$, Assertions~3--5 and Step~2 of the proof of Lemma~3.2
give at most one zero. On the discriminant locus the factor $(2X-\alpha_{\mathrm{L}})^2$ from
formula~(3.18) vanishes at the stationary inflection, where $Q=0$. The opposite root lies
strictly between extrema with $Q\ne0$, so this factor creates no extra zero. For four folds,
\eqref{eq:LL-normal} gives the pair $\sigma_1,\sigma_2$ about $a_{\mathrm{L}}$ in shape~A of
\cite[Figure~12]{LiLu}. Assertions~7, 11 and~12 of the proof of Lemma~3.3
give the bound, including the partition curves.

For (b), the same normalization gives shape~D in region~(3), with moving endpoints
$\sigma_1,\sigma_4$ and the fixed jump at $X=1$. The conclusion then follows from~\cite[Lemma~3.3, Assertion~13]{LiLu}.

For (c), take $p_*=p_3$, $a_{\mathrm{L}}=p_1/p_3$ and $b_{\mathrm{L}}=p_2/p_3$.
This gives polynomial~(3.7) of~\cite{LiLu}. Write $h_{\mathrm{L}}=\Phi'/X$.
The roots correspond to $\sigma_3(X)\in(b_{\mathrm{L}},1)$ and $X<0$.
If $\Phi(b_{\mathrm{L}})>0$, Step~2 of the proof of Lemma~3.2, p.~4455,
gives $h_{\mathrm{L}}(\sigma_3(X))<h_{\mathrm{L}}(X)$.

It remains to consider $\Phi(b_{\mathrm{L}})=0$, that is, $F(-p_2)=0$.
Put $p=p_2$. Since $0$ and $-p$ are double zeros of $F$, we have
\[
F(x)=a_5x^2(x+p)^2(x+d).
\]
Moreover, $F''(-p)=2a_5p^2(d-p)>0$, so $d>p$.

Let $F(x^-)=F(x^+)=h>0$, with $x^-\in(-p_3,-p)$ and $x^+>0$.
At either root, $x(x+p)>0$ and $x+d>0$. Using $Q(x)=F'(x)/x$
and the factorization above, we obtain
\[
\frac{Q(x)}{2\sqrt{a_5h}}
=
\left(2+\frac{p}{x}\right)\sqrt{x+d}
+\frac{x+p}{2\sqrt{x+d}}.
\]
At $x=x^-$, we have $1<2+p/x^-<2$, $0<x^-+d<d$ and $x^-+p<0$,
so the right-hand side is less than $2\sqrt d$.
At $x=x^+$, the first term exceeds $2\sqrt d$ and the second is
positive. Thus $Q(x^-)<Q(x^+)$, as required.

For (d), we verify the claim directly without invoking~\cite{LiLu}. Since $F'(x)=x(x+2)(x-3)^2$, we
have $Q(x)=x^3-4x^2-3x+18$ and $F(x)=\frac15x^5-x^4-x^3+9x^2$. Suppose $J'(h)=0$ for some $0<h<A$, and
set $l=l(h)<0<r=r(h)$, $u=l+r$, and $v=lr$. Dividing $Q(l)-Q(r)=0$ by $l-r\ne0$ gives $u^2-v-4u-3=0$,
that is, $v=u^2-4u-3$. Substituting this into the equal-height condition $F(l)=F(r)$ yields
\[
0=\frac{F(l)-F(r)}{l-r}=-\frac15(u-6)(u-1)^3.
\]
If $u=6$, then $v=9$ and hence $l=r=3$; if $u=1$, then $v=-6$ and hence $(l,r)=(-2,3)$, corresponding
to $h=A$. Neither case occurs for $0<h<A$, so $J'$ has no zero in $(0,A)$.
\end{proof}

\section{A computer-assisted inequality}\label{app:cert}

Put $m=s/(1+s)$ and $\gamma=5a_5(1+s)>0$. Replacing $F$, $Q$ and $\cH$ by $F/\gamma$, $Q/\gamma$ and
$\cH/\gamma^2$, and keeping the same letters for the normalized functions, we have
\[
Q(x)=\bigl(m+(1-m)x\bigr)(x-1)(x-t),\qquad F'=xQ,\qquad \cH'=xQ^2,\qquad F(0)=\cH(0)=0 .
\]
This normalization preserves equal-height pairs and the sign of slow divergence integrals.

\begin{proposition}\label{prop:cert}
In the four-fold case with $A\ge B\ge C\ge0$, the junction value of Section~\ref{sec:five} satisfies,
in the normalization above,
\[
I_3(B)=I_4(B)=\cH(l(B))-\cH(r(C))+\cH(t)-\cH(u(B))>\frac1{6000}.
\]
In the original normalization, $I_4(B)>\gamma^2/6000>0$.
\end{proposition}

To see which parameters have to be covered, put
\[
D(s,t)=3s^2+(5t-4)s+3-5t,\qquad E(m,t)=5m(2-t)-(1-m)t(3t-5).
\]
Direct integration of the normalized polynomial gives
\[
A-B=\frac{(1-m)(1+s)^3}{60}\,D(s,t),\qquad C=\frac{t^3}{60}\,E(m,t).
\]
For $t\ge1$ we have $\partial_sD=6s+5t-4>0$ and $D(2/3,t)=5(1-t)/3\le0$, so $A\ge B$ implies
$s\ge2/3$, that is, $m\ge2/5$. Since $m<1$, we have $E(m,t)<0$ for $t\ge2$, so $C\ge0$ implies
$t<2$. Hence it suffices to consider $(m,t)\in[2/5,1]\times[1,2]$.

On a box $[m_1,m_2]\times[t_1,t_2]$, indicate the dependence on $(m,t)$ by subscripts. At the four
vertices of the box we check $F_{m,t}(-z)\le F_{m,t}(1)\le F_{m,t}(U)$ for rational witnesses $z\ge0$
with $(1-m_1)z\le m_1$, and $U\ge t_2$. Since $F_{m,t}(x)$ is affine in $m$ and in $t$, both
inequalities extend to the whole box and give $-l(B)\ge z$ and $u(B)\le U$. A third rational witness
$v\in[0,1]$ satisfies
\[
F_{m_2,t_1}(1)-F_{m_2,t_1}(t_1)\ge F_{m_2,t_2}(1)-F_{m_2,t_2}(v).
\]
Since $B-C=\int_1^tx(x-1)(t-x)[m+(1-m)x]\,dx$ attains its minimum on the box at $(m_2,t_1)$ and
$B-F(v)$ attains its maximum at $(m_2,t_2)$, this gives $r(C)\le v$. As $m+(1-m)x\ge0$ increases
with $m$ on $[-z,0]\cup[0,v]$ and decreases with $m$ on $[t,U]$, bounding $(x-t)^2$ on each interval
(on $[t,U]$ we first use $0\le m+(1-m)x\le m_1+(1-m_1)x$ and $0\le x-t\le x-t_1$, and only then
extend the nonnegative majorant to $[t_1,U]$) yields
\[
I_4(B)\ge\cH_{m_1,t_1}(-z)-\cH_{m_2,t_2}(v)-\bigl[\cH_{m_1,t_1}(U)-\cH_{m_1,t_1}(t_1)\bigr],
\]
and the right-hand side is checked to exceed $1/6000$ in exact rational arithmetic.

Boxes outside the region $A\ge B\ge C\ge0$ are discarded by means of the necessary conditions
$D(s,t)\ge0$ (for $A\ge B$) and $E(m,t)\ge0$ (for $C\ge0$). If $m_2<1$, put $s_2=m_2/(1-m_2)$;
since $D$ increases in $s$ and is affine in $t$, the whole box satisfies $A<B$ if
$\max\{D(s_2,t_1),D(s_2,t_2)\}<0$. (Note that $\partial_tD=5(s-1)$ changes sign at $s=1$, so both
endpoints in $t$ are needed.) This exclusion is not used when $m_2=1$. On the verification
rectangle $\partial_mE=3t^2-10t+10>0$ and $\partial_tE=-5m+(1-m)(5-6t)<0$, so $E(m_2,t_1)<0$ implies
$C<0$ throughout the box. Neither exclusion discards a box containing a parameter pair for which both
$A\ge B$ and $C\ge0$ hold (a box may touch one of the curves $A=B$, $C=0$ and still be discarded
by the other condition). On the boundary of the region, $I_4(B)$ denotes the continuous extension
of the displayed expression from $A>B>C>0$. The supplementary dyadic cover (\texttt{boxes.json}) has depth at most six, with rational
witnesses on a grid of mesh $1/(5\cdot2^{13})$. Of its $217=173+30+14$ leaves, $173$ satisfy the
certified integral bound, $30$ are excluded by $A<B$, and $14$ are excluded by $C<0$; the smallest
certified lower bound exceeds $1/6000$ by about $7.0\times10^{-7}$. The accompanying Python script
(\texttt{verify\_boxes.py}) checks every leaf using exact rational arithmetic from the standard
library.

\section*{Data availability}
The data and code used for the computer-assisted verification in \ref{app:cert} (the dyadic box
list \texttt{boxes.json} and the exact rational verification script \texttt{verify\_boxes.py}) are
provided as supplementary material accompanying this article.

\section*{Concluding remark}

In our October 2026 preprint
(\href{https://arxiv.org/abs/2610.13080v1}{arXiv:2610.13080v1}),
we proved that, for each fixed quintic $F$, system~\eqref{eq:system}
has at most two limit cycles when $\lambda=0$ and $\eps>0$ is
sufficiently small. Theorem~\ref{thm:main} extends this bound to the
joint limit $(\eps,\lambda)\to(0^+,0)$ at arbitrary relative rates
and counts limit cycles with multiplicity.

A contemporaneous preprint by OpenAI~\cite{OpenAI2026Quintic} claims
a sharp upper bound of two geometrically distinct limit cycles for
the classical Li\'enard system~\eqref{eq:lienard} with $\deg F\le 5$.
By translation and rescaling, the asserted bound implies the geometric
count in Theorem~\ref{thm:main}. As of November 2026, we know of no
peer-reviewed publication of~\cite{OpenAI2026Quintic} or published
independent expert verification of the analytic proof given there.
The accompanying
\href{https://github.com/openai/math/blob/main/lean/docs/143.md}{Lean formalization}
addresses the geometric count, with each isolated periodic orbit
counted once.

Our proof is independent of~\cite{OpenAI2026Quintic}. Both arguments
compare the two half-orbits. The preprint fits quadratic profiles to
half-orbit data to determine when the two halves match. In the
slow--fast regime, we explain why a third cycle is excluded by ordering
the limiting configurations and determining the stability of nearby
cycles from the sign of the slow divergence integral. When the integral
vanishes, we compare the forward and backward transition maps. This
comparison determines the multiplicity at a unit multiplier and, together
with the global ordering, shows that any nonhyperbolic periodic orbit
in this regime is unique and has multiplicity exactly two.

\section*{Acknowledgments}
This work is supported by NSF of China (Numbers 12571190, 12322109) and Hunan Basic Science
Research Center for Mathematical Analysis (2024JC2002).

\end{document}